\documentclass[draftcls,12pt,onecolumn]{IEEEtran}
\usepackage{amsmath, amsfonts,amssymb, latexsym,epsfig,color,rotating,
times,float,subfigure,algorithm,verbatim,framed}

\newcommand{\mz}{m_z}
\newcommand{\Q}{\mathcal{Q}}
\newcommand{\arginf}{\operatorname{arg\, inf}}
\newcommand{\eig}{\lambda}
\newcommand{\price}{\mathbf{\nu}}
\newtheorem{theorem}{Theorem}
\newtheorem{lemma}{Lemma}
\newtheorem{definition}{Definition}
\newtheorem{proof}{Proof}
\newcommand{\nablat}{\widehat{\nabla}_{\phi}}
\newcommand{\beq}{\begin{equation}}
\newcommand{\reals}{\mathbb{R}}
\renewcommand{\d}{\mathbf{d}}
\newcommand{\p}{T}

\newcommand{\lstar}{{l^*}}

\newcommand{\Rs}{\mathcal{S}_\text{stop}}
\newcommand{\Fa}{\mathbf{F}}

\newcommand{\bt}{\underline{\theta}}
\newcommand{\bP}{\bar{P}}
\newcommand{\plat}{\xi}
\newcommand{\ol}{\preceq_l}
\newcommand{\og}{\succeq_l}

\newcommand{\Vb}{\bar{V}}

\newcommand{\pdf}{\mathcal{M}}
\newcommand{\eeq}{\end{equation}}
\newcommand{\qed} {{$\hfill\blacksquare$}}
\newcommand{\E}                 {\Bbb{E}}
\renewcommand{\P}                 {\Bbb{P}}
\newcommand{\Mu}{\boldsymbol{\mu}}
\newcommand{\Ep}{\E^\mu}
\newcommand{\Ept}{\E^{\mu_\theta}}

\newcommand{\xr}{\delta_x}
\newcommand{\yr}{\delta_y}
\newcommand{\zr}{\plat_z}

\newcommand{\cop}{c_{\price}}
\def\L{{\cal L}}
\DeclareMathOperator*{\argmin}{arg\,min}
\DeclareMathOperator*{\argmax}{arg\,max}

\newcommand{\gr}{\succeq}
\newcommand{\gs}{\succ}
\newcommand{\lr}{\preceq}

\newcommand{\R}{\mathcal{R}}
\newcommand{\Cb}{\bar{C}}
\newcommand{\bs}{\bar{s}}
\title{
Sequential Detection with Mutual Information Stopping Cost}
\author{Vikram Krishnamurthy {\em Fellow IEEE}, Robert Bitmead {\em Fellow IEEE},  \\ Michel Gevers  {\em Fellow IEEE} and Erik Miehling\thanks{{Vikram Krishnamurthy 
(vikramk@ece.ubc.ca) and Erik Miehling (erikm@ece.ubc.ca) are with the Department of Electrical and Computer Engineering, University of British Columbia, Vancouver, BC, V6T 1Z4, Canada. Robert Bitmead (rbitmead@ucsd.edu) is with the Department of Mechanical and Aerospace  Engineering, University
of California San Diego, CA 92093-0411, USA.  Michel Gevers (gevers@csam.ucl.ac.be)  is with the Department of Mathematical Engineering, Universite Catholique de Louvain, Louvain-la-Neuve, Belgium. 

{The contribution of the third author Gevers was limited to the proofs of the submodularity properties in the Appendix. That} of the fourth  author Miehling was to code 
the algorithms proposed in the paper  and prepare the numerical examples in Section \ref{sec:numerical}.The work of the first and fourth authors was supported by a NSERC Strategic Grant and DRDC Ottawa. 
The work of the third author was supported by the Belgian Network DYSCO (Dynamical Systems, Control, and Optimization), funded by the Interuniversity Attraction Poles Programme, initiated by the Belgian State, Science Policy Office. The scientific responsibility rests with the authors.   
}} }

\begin{document}
%
\maketitle
\begin{abstract}
This paper formulates and solves a sequential detection problem that involves
the mutual information (stochastic observability)  of a Gaussian process observed in noise with missing measurements.
The main result is that  the optimal decision 
is characterized by a monotone policy on the partially ordered set of  positive definite covariance matrices.
This monotone structure  implies that numerically efficient algorithms  can be designed to estimate and
implement  monotone 
parametrized  decision policies.
The sequential detection problem is motivated by applications in radar scheduling
where the aim is to maintain the mutual information of all targets within a specified bound.
We illustrate the problem formulation and  performance of monotone parametrized policies via numerical examples in fly-by and
persistent-surveillance applications involving  a GMTI (Ground Moving Target Indicator) radar.
\end{abstract}

\section{Introduction}
\label{sec:intro}

Consider the following sequential detection problem. $L$  targets (Gaussian processes) are 
allocated priorities $\price_1,\price_2,\ldots,\price_L$.  A sensor obtains measurements of these $L$
evolving targets with signal to noise ratio (SNR) for target $l$ proportional to priority $\price_l$. A decision maker has two choices at each time $k$: If the decision maker chooses action $u_k=2$ (continue) then the sensor takes another measurement and accrues a measurement cost $\cop$.
If the decision maker chooses action $u_k=1$ (stop), then a stopping cost proportional to the mutual information 
(stochastic observability) of the targets is accrued and the
problem terminates. 
What is the optimal time for the decision maker to apply the stop action?
 Our  main result 
 is  that the optimal decision policy is a monotone function
of the target covariances (with respect to the positive definite {partial} ordering). This  facilitates devising numerically efficient algorithms
to compute the optimal policy.

 The sequential detection problem addressed in this paper is non-trivial since the decision
 to continue or stop
 is based on Bayesian estimates of the targets' states. In addition to Gaussian noise in the measurement
 process, the sensor has a non-zero
probability of missing observations.
 Hence, the sequential detection problem is  a partially observed 
stochastic control problem.   
Targets with high priority are observed with higher SNR and the uncertainty (covariance) of their estimates
decreases. Lower priority targets are observed with lower SNR  and their relative uncertainty increases. The  aim is  to devise a sequential detection policy that maintains the stochastic observability (mutual information or conditional entropy) of all targets within a specified bound.

{\em Why stochastic observability?}  As mentioned above, the stopping cost in our sequential detection problem is a function of the mutual information
(stochastic observability) of the targets.
The use of mutual
information as a measure of stochastic observability
was originally investigated in \cite{MH88}. In \cite{LI99}, determining optimal observer trajectories 
to maximize the stochastic observability of a single target is formulated 
as a stochastic dynamic programming problem -- but no structural
results or characterization of the optimal policy is given;
see also \cite{LB11}. We also refer to 
\cite{GL09} where a nice formulation of sequential waveform design for MIMO radar is given  using a Kullback-Leibler divergence based approach.
 As described  in Section \ref{sec:existence}, another {favorable} property
of stochastic observability is that its monotonicity with respect to covariances does not require stability of the state matrix of the target
(eigenvalues strictly inside the unit circle). 
In target models, the state matrix for the dynamics of the target has eigenvalues at 1 and thus is not stable. 

\noindent {\bf Organization and Main Results}:\\
(i) To motivate the sequential detection problem, Section \ref{sec:dynamics}  presents a GMTI (Ground moving target indicator)
radar  with macro/micro-manager architecture and 
{a} linear Gaussian state space model for the dynamics of each target. A Kalman filter is  used
to track each target over the time scale at which the micro-manager operates. 
Due to the {presence of missed detections},
the covariance update via the Riccati equation is  measurement dependent 
(unlike the standard Kalman filter
where the covariance is functionally independent of the measurements).
\\
(ii) In Section \ref{sec:seqdet}, the sequential detection problem 
 is formulated. 
The cost of stopping is the stochastic observability which is based on the mutual information of the targets. 
The optimal decision policy satisfies Bellman's  dynamic programming equation. However,   it is  not
possible to compute the optimal policy in closed form.\footnote{For stochastic control problems with continuum state spaces such as 
considered in this paper,  apart from special cases such as linear quadratic control and partially observed Markov
decision processes, there are no finite dimensional characterizations of the optimal policy \cite{Ber00b}.  Bellman's equation does not translate into
practical solution methodologies since the state space is a continuum. Quantizing  the space of covariance matrices  to a finite state space and then
 formulating the problem as a finite-state Markov decision process is {infeasible} since  such quantization
typically would require an intractably large state space.}  Despite this,
our main  result (Theorem \ref{thm:main}) shows that 
 the optimal  policy is a monotone function of the target covariances. This result
 is  useful for two reasons:
 (a) Algorithms can be designed to construct
 policies that satisfy this monotone structure. 
 (b) The monotone structural result holds without stability assumptions on the linear dynamics.
 So there is an inherent
 robustness  of this result since it holds even if the underlying  model parameters are 
not exactly specified.\\
  (iii)   Section \ref{sec:approx} exploits the monotone structure of the optimal decision policy to construct finite dimensional parametrized policies.
Then a simulation-based stochastic approximation (adaptive filtering) algorithm (Algorithm \ref{alg1}) is given to compute  these optimal parametrized policies.
The practical
implication is that,  instead of solving an intractable dynamic programming problem, we exploit the monotone structure of the optimal
policy to compute such  parametrized policies in polynomial time.
\\ 
(iv) 
 Section \ref{sec:numerical}  presents a detailed application of the sequential detection
 problem in GMTI radar resource management.
    By bounding the magnitude of the nonlinearity in the GMTI measurement model, we show that for typical operating values, the system can be approximated
by a linear time invariant state space model.
Then  detailed
 numerical examples are given that use the above monotone policy and stochastic approximation
 algorithm to {demonstrate} the performance of the radar management algorithms.
We 
present numerical results for  two important GMTI surveillance problems, namely, the target fly-by problem and the persistent surveillance problem.  In both cases, detailed numerical examples are given and the  performance is compared
 with  periodic stopping policies.
  Persistent surveillance has received much attention in the defense literature 
  \cite{BDW10,Whi10}, since it can  provide critical, long-term surveillance information. By tracking targets for long periods of time using aerial based radars, such as DRDC-Ottawa's XWEAR radar \cite{BDW10} or the U.S. Air Force's  Gorgon Stare Wide Area Airborne Surveillance System, operators can ``rewind the tapes'' in order to determine the origin of any target of interest \cite{Whi10}. 
   \\
{(v)} The appendix  presents the proof of Theorem \ref{thm:main}.
It uses lattice programming and supermodularity. A crucial step in the proof is that the conditional
entropy described by the Riccati equation update is monotone.
This involves use of Theorem \ref{thm:bitmead} which derives monotone properties of the Riccati and Lyapunov
equations.
The idea of using lattice programming and supermodularity to prove the existence of monotone
policies is well known in stochastic control, see \cite{HS84} for a textbook treatment of the countable state Markov decision
process case. However, in our case since the state space comprises  covariance matrices  that are only partially ordered, the optimal policy is monotone
with respect to this partial order. 
The structural results of this paper allow us to determine the nature of the optimal
policy without brute force numerical computation.

{\em Motivation -- GMTI Radar Resource Management}: This paper is motivated by GMTI radar resource management problems
 \cite{BP99,WK06,KD09}.
The radar  {macro-manager} deals with 
priority allocation of targets, determining regions to scan, 
 and target revisit times. 
The radar {\em micro-manager} controls the target  tracking algorithm
and determines how long to maintain a priority allocation set by the macro-manager.  In the context of GMTI radar 
micro-management,
the sequential detection problem outlined above reads: Suppose 
 the radar macro-manager specifies a particular target priority allocation.  How long should the micro-manager track targets using the current
priority allocation  before returning control to the macro-manager? Our main result, that the optimal decision policy is a monotone
function of the targets' covariances, facilitates devising numerically efficient algorithms for the optimal radar micro-management policy.

\section{Radar Manager Architecture and Target Dynamics}
\label{sec:dynamics}

This section  motivates the sequential detection problem by outlining
 the macro/micro-manager architecture of the GMTI radar and target dynamics.
 (The linear dynamics of the target model are justified in Section \ref{sec:appgmti} where
a detailed description is given of the GMTI kinematic model). 

\subsection{Macro- and Micro-manager Architecture} \label{sec:architecture}
(The reader who is {uninterested} in the
radar application can skip this subsection.)
Consider
a GMTI radar with an agile  beam tracking  $L$ ground moving targets indexed by  $l \in \{1,\ldots,L\}$.
In this section we describe a two-time-scale radar management scheme comprised of a micro-manager
and a macro-manager.

\paragraph{Macro-manager} At the beginning
of each scheduling interval $n$, the radar macro-manager allocates the target priority vector $\price_n = (\price^1_n,\ldots,\price^L_n)$. Here the priority
of target $l$ is $\price_n^l \in [0,1]$ and $\sum_{l=1}^L \price_n^l = 1$. The priority weight $\price_n^l$ determines what resources the radar devotes to 
target $l$. This affects the track variances as  described below.
 The choice $\price_n$  is  typically  
 rule-based, depending on several  extrinsic factors.   For example, in  GMTI radar systems, the
 macro-manager 
picks the target priority vector $\price_{n+1}$ based on the track variances (uncertainty) and
 threat levels of the $L$ targets.
The track variances of the $L$ targets are determined by the Bayesian tracker as discussed below.

\paragraph{Micro-manager} Once  the target priority vector $\price$ is chosen (we omit the subscript  $n$ for convenience), the micro-manager  is initiated.
 The clock on  the  fast time scale $k$ (which is called the decision epoch time scale in Section
 \ref{sec:appgmti})
 is reset to  $k=0$ and commences
ticking.
At this decision epoch time scale, $k=0,1,\ldots$, the $L$ targets are tracked/estimated by a Bayesian tracker.
Target $l$ with priority $\price^l$ is allocated the fraction 
$\price^l$ of the total number of observations (by integrating $\price^l \Delta$ observations on the fast time scale, see 
Section \ref{sec:appgmti})
 so that the observation noise
variance is scaled by $1/(\price_l \Delta)$.
The question we seek to answer is:
{\em How long should the micro-manager track
the $L$ targets with priority vector $\price$  before
returning control to 
 the macro-manager to pick
a new priority vector}? We formulate this as a sequential decision problem.

Note that the priority allocation vector $\price$ and track variances of the $L$ targets
 capture the interaction between the micro- and macro-managers.

\subsection{Target Kinematic Model and Tracker} \label{sec:kinematic}
We now describe the target kinematic model at the epoch time scale $k$:
Let  $s^{l}_{k} =  \left[x^{l}_{k}, \dot{x}^{l}_{k}, y^{l}_{k}, \dot{y}^{l}_{k}\right]^{T}$ denote the Cartesian coordinates and velocities  of the ground moving target $l \in \{1,\ldots,L\}$.
Section \ref{sec:appgmti} shows that on the
 micro-manager time scale,  the GMTI target dynamics can be approximated
  as the following  linear time invariant Gaussian state space model
\begin{align} \label{eq:gauss}
{s}_{k+1}^l  &= F {s}_k^l + G w_k^l, \nonumber\\
{z}_{k}^l &= \begin{cases} H  {s}_k^l +  \frac{1}{\sqrt{\price^l \Delta}}  v_k^l,    &  \text{ with probability $p_d^l$,}\\
\emptyset, &  \text{ with probability $1- p_d^l$.}  \end{cases}
\end{align}
The parameters $F$, $G$, $H$ are defined in Section \ref{sec:numerical}. They can  be target  ($l$) dependent; to simplify notation we have not done this.
In (\ref{eq:gauss}),  $z^{l}_{k}$ denotes a 3-dimensional observation vector of target $l$ at epoch time $k$.
 The noise processes
$w^{l}_{k}$ and $v^{l}_{k}/\sqrt{\price^l \Delta}$ are mutually independent, white, zero-mean Gaussian random vectors with covariance matrices $ Q^{l} $ and $ R^{l}(\price^l) $, respectively. ($Q$ and $R$ are defined in Section \ref{sec:numerical}).
Finally, $p_d^l$ denotes the probability of detection of target $l$, and $\emptyset$ represents a missed observation that contains no information about state $s$.\footnote{With suitable notational abuse, we use `$\emptyset$' as a label to denote a missing observation.
When a missing observation is encountered, the track estimate is updated by the Kalman predictor with covariance update (\ref{eq:liapunov}).}

Define the 
one-step-ahead predicted covariance matrix of target $l$ at time $k$ as
$$
P^l_{k} = \E\biggl\{ \bigl( s_k^l - \E\{s_k^l | z^l_{1:k-1}\} \bigr) \bigl( s_k^l - \E\{s_k^l | z^l_{1:k-1}\} \bigr) ^\p \biggr\}.
$$ 
Here the superscript $T$ denotes  transpose.
Based on the priority vector $\price$ and model (\ref{eq:gauss}),  the covariance 
of the state estimate of target $l \in \{1,\ldots,L\}$ is 
computed via the following  measurement  dependent Riccati equation
\beq
\label{eq:riccEqn}
 P^l_{k+1} =   \mathcal{R}(P^l_{k},z_k) 
{\buildrel\rm	def\over=}  FP^l_{k}F^{T} + Q^l 
- I(z_k^l\neq \emptyset)  FP^l_{k}H^{T}\bigl(HP^l_{k}H^{T}+R^l(\price^l)\bigr)^{-1}HP^l_{k}F^{T} .
\eeq
Here $I(\cdot)$ denotes the indicator function.
In the special case when a target $l$ is allocated zero priority ($\price^{(l)} = 0$), or when there is a missing observation
($z_k^l = \emptyset$), 
then 
(\ref{eq:riccEqn}) specializes to  the Kalman predictor updated via the Lyapunov equation
\beq P^l_{k|k-1} = \mathcal{L}(P^l_{k}) {\buildrel\rm	def\over=}  FP^l_{k-1}{F}^{T} + Q^l 
\label{eq:liapunov} .\eeq

\section{Sequential Detection Problem}
\label{sec:seqdet}

This section presents our main structural result on the sequential detection problem.
Section \ref{sec:formu} formulates the stopping cost in terms of the mutual information 
of the targets being tracked.
Section \ref{sec:seqdetf} formulates the sequential detection problem. The optimal decision policy is expressed as the solution of a stochastic dynamic programming problem.
The main result (Theorem \ref{thm:main} in Section \ref{sec:existence})  states that the optimal policy is a monotone function of  the target covariance. As a result, the optimal policy can be parametrized by monotone policies and estimated in a computationally
efficient manner via stochastic approximation (adaptive filtering) algorithms. This is described in Section \ref{sec:approx}.

{\em Notation}: Given the priority vector $\price$ allocated by the macro-manager,
let $a \in \{1,\ldots , L\}$ denote the highest priority target, i.e,  $a = \argmax _l \price^l$.
Its covariance is denoted $P^a$.
We use the notation $P^{-a}$ to denote the set of covariance matrices of the remaining $L-1$ targets.
The sequential decision problem  below is formulated in terms of  $(P^a,P^{-a})$.

\subsection{Formulation of  Mutual Information Stopping Cost} \label{sec:formu}
As mentioned in Section \ref{sec:architecture},  once the radar macro-manager  determines 
the priority vector $\price$, the micro-manager switches on and its clock $k=1,2,\ldots$ begins to tick.
The radar micro-manager then solves a sequential detection problem involving two actions: 
At each slot $k$, the micro-manager chooses action $u_k \in \{\text{1 (stop) }, \text{2 (continue) }\}$.
To formulate the sequential detection problem, this subsection specifies the costs incurred with
these actions.

{\em Radar Operating cost}: If the micro-manager  chooses action $u_k=2$ (continue), it incurs the  radar operating cost denoted as $\cop$. Here $\cop>0$ depends on the radar operating parameters,

{\em Stopping cost -- Stochastic Observability}: If the micro-manager chooses
action 
  $u_k=1$ (stop), a stopping cost is incurred. In this paper, we formulate a stopping
  cost in terms of the stochastic observability
  of the targets, see also \cite{MH88,LI99}. Define the stochastic observability of each  target 
  $l \in \{1,,\ldots,L\}$ as  the mutual information
  \beq
  I(s^l_k; z_{1:k}^l) = \alpha^l h(s_k^l) - \beta^l h(s^l_k|z_{1:k}^l). \label{eq:mutual}\eeq
     In  (\ref{eq:mutual}), $\alpha^l$ and $\beta^l$ are non-negative constants chosen by the designer.
  Recall from  information theory \cite{CT06}, that $h(s_k^l)$ denotes
  the differential entropy of 
 target $l$ at time $k$. Also  $ h(s^{l}_k | z^{l}_{1:k})$ denotes the conditional differential
 entropy  of target $l$ at  time $k$ given the observation history
$z^{l}_{1:k}$. The mutual information $I(s^l_k;z_{1:k}^l) $ is the average reduction in uncertainty
of the target's coordinates $s_k^l$  given measurements $z_{1:k}^l$.
In the standard definition of mutual information $\alpha^l = \beta^l = 1$. However, we are also
interested in the special case when $\alpha^l=0$, in which case, we are considering the conditional entropy
for each target (see Case 4 below).

Consider the following stopping cost if the micro-manager chooses action $u_k=1$ at time $k$:
\beq \label{eq:stopp}
\Cb(s_k,z_{1:k}) = - I(s_k^a;z_{1:k}^a) + \Fa(\{I(s_k^{l},z_{1:k}^{l}); l \neq a\}) .\eeq
Recall $a$ denotes the highest priority target.
In (\ref{eq:stopp}), $\Fa(\cdot)$ denotes a function chosen by the designer {to be}   monotone increasing  in each of its $L-1$ variables (examples are given below).

The following lemma follows from straightforward arguments in \cite{CT06}.

\begin{lemma}  \label{lem:gauss}Under the assumption of linear  Gaussian dynamics  (\ref{eq:gauss}) for each target $l$, the mutual
information of target $l$ defined in (\ref{eq:mutual}) is
\begin{align}
I(s^l_k,z_{1:k}^l) &= \alpha^l\log|\bP_{k}^l| - \beta^l \log |P_{k}^l|, \label{eq:mutualinfo}
\end{align}
where $\bP_k^l  = \E\{ (s_k^l - \E\{s_k^l\})(s_k^l - \E\{s_k^l\})^\p\},$ 
$P_k^l =  \E\{ (s^l_k - \E\{s^l_k|z_{1:k}\})(s_k^l - \E\{s_k^l|z_{1:k}\})^\p\}.$
Here $\bP_k^l$ denotes the predicted (a priori) covariance of target $l$ 
at epoch $k$
given no observations. It is  computed
using  the Kalman predictor covariance update (\ref{eq:liapunov}) for $k$ iterations. Also,  $P_k^l$  is the  posterior covariance and is computed via the Kalman filter
covariance update (\ref{eq:riccEqn}).
\qed \end{lemma}

Using Lemma \ref{lem:gauss}, the stopping cost  $\Cb(\cdot,\cdot)$ in (\ref{eq:stopp}) can
be expressed in terms of the Kalman filter and predictor
covariances. Define the four-tuple of {sets of} covariance matrices
\beq P_k = (P_k^a,\bP_k^a,P_k^{-a},\bP_k^{-a}) .\label{eq:4tuple}\eeq
Therefore the stopping cost (\ref{eq:stopp}) can be expressed as
\beq \label{eq:stop}
\Cb(P_k) =   -\alpha^a \log|\bP_k^a|
+  \beta^a \log|P_k^a| + \Fa \left(\{\alpha^l \log|\bP_k^l|
-  \beta^l \log|P_k^l| ;\; l \neq a\right\}) .\eeq

{\bf Examples}:  We consider the following examples of  $\Fa(\cdot)$ in (\ref{eq:stop}) :

\noindent{\em Case 1. Maximum mutual information difference stopping cost}: 
$\Cb(s_k,z_{1;k})= 
-I(s_k^a,z_{1:k}^a)+\max_{l \neq a} I(s_k^{l},z_{1:k}^{l}) $ in which case,
\beq \Cb(P_k) =  -\alpha^a \log|\bP_k^a|
+  \beta^a \log|P_k^a| + \max_{l \neq a} \left[\alpha^l \log|\bP_k^l|
-  \beta^l \log|P_k^l| \right] .\label{eq:maxp}\eeq
The stopping cost is the difference in mutual information between the target 
with highest mutual information and the target with highest priority. This can be viewed as a stopping cost that
discourages stopping too soon.

\noindent {\em Case 2. Minimum  mutual information difference stopping cost}: 
 $\Cb(s_k,z_{1;k})= -I(s_k^a,z_{1:k}^a)+\min_{l \neq a} I(s_k^{l},z_{1:k}^{l}) $
 in which case, 
 \beq \Cb(P_k) =  -\alpha^a \log|\bP_k^a|
+  \beta^a \log|P_k^a| + \min_{l \neq a} \left[\alpha^l \log|\bP_k^l|
-  \beta^l \log|P_k^l| \right]. \label{eq:minp}\eeq
The stopping cost is the difference in mutual information between  the target 
with lowest mutual information and the target with highest priority. This can be viewed as a conservative
stopping cost in the sense that preference is given to stop sooner.

\noindent {\em Case 3. Average mutual information difference stopping cost}: 
$\Cb(s_k,z_{1;k})= -I(s_k^a,z_{1:k}^a)+\sum_{l \neq a} I(s_k^{l},z_{1:k}^{l}) $  in which case,
\beq \Cb(P_k)=
-\alpha^a \log|\bP_k^a|
+  \beta^a \log|P_k^a| + \sum_{l \neq a} \left[\alpha^l \log|\bP_k^l|
-  \beta^l \log|P_k^l| \right]. \label{eq:avgp}
\eeq
This stopping cost is the difference between  the average mutual information of the $L-1$ targets
(if $\alpha^l$ and $\beta^l$ include a $1/(L-1)$ term)  and the
highest priority target. 

\noindent {\em Case 4. Conditional differential entropy difference stopping cost}: We are also interested in the following special case which involves scheduling between a Kalman filter and $L-1$ measurement-free Kalman predictors,
see \cite{EKN05}.
Suppose the high priority target $a$ is allocated a  Kalman filter and the remaining $L-1$  targets are allocated measurement-free Kalman predictors. This corresponds 
to the case where  $\price^a = 1$ and $\price^l  =0$ for $l\neq a$ in (\ref{eq:gauss}),  that is,
the radar assigns all its resources to target $a$ and no resources to any other target.
 Then solving the sequential detection problem is
equivalent to posing the following question: What is the optimal stopping time $\tau$ when  the radar should decide to start tracking
another target?
In this case,   the mutual information of each target $l\neq a$ is zero
(since $\bP_k^l = P_k^l$ in (\ref{eq:mutualinfo})). So 
it is appropriate to choose  $\alpha^l = 0$ for $l\neq a$ in (\ref{eq:stop}). Note
from (\ref{eq:mutual}), that when $\alpha^l = 0$, the stopping cost of each individual target becomes  the
negative of its conditional entropy.
That is, the  stopping
cost is the difference in the conditional differential entropy instead of the mutual information.

{\em Discussion}: A natural question is: {\em How to pick 
the stopping cost (\ref{eq:stop}) depending on the  target priorities?}
One can design the choice of stopping cost (namely, Case 1, 2 or 3 above) depending on the  range
of target priorities. For example, suppose the priority of a target is the negative
of its mutual information.\\
(i) If two or more targets have similar
high priorities, it makes sense to use Case 2 since the stopping cost $\Cb(P)$ would be close to zero.
This would give incentive for the micro-manager to stop quickly and consider other high priority targets.
Note also that if multiple targets have  similar high priorities,
the radar would devote similar amounts  of time to them according to the protocol in Section 
\ref{sec:architecture}, thereby not compromising the accuracy of the estimates of these targets.
\\
(ii) If  target $a$ has a significantly higher priority than
all other targets, then Case 1 or 3 can be chosen for the stopping cost. As mentioned above, Case 1 would discourage stopping
too soon thereby allocating more resources to target $a$. In comparison, Case 3 is a compromise between
Case 1 and Case 2, since it would consider the average of all other target priorities (instead of the
maximum or minimum).
\\
Since, as will be shown in Section \ref{sec:approx}, the parametrized micro-management policies can be
implemented efficiently, the radar system can switch between the above stopping costs in real time
(at the macro-manager time scale).
Finally,  from a practical point of view, the macro-manager, which is responsible for assigning the priority allocations, will rarely assign equal priorities to two targets. This is due to the fact that the priority computation in realistic scenarios is based on many factors such as target proximity and heading relative to  assets in surveillance region, error covariances in state estimates, and target type.

\subsection{Formulation of Sequential Decision Problem}\label{sec:seqdetf}
With the above stopping and continuing costs, we are now ready to formulate the sequential detection problem
that we wish to solve. Let $\mu$ denote a  stationary decision policy of the form  
\beq \mu: P_k\rightarrow u_{k+1} \in \{1 \text{ (stop) },  2 \text{ (continue) }\} . \eeq
Recall from (\ref{eq:4tuple})  that $P_k$  is a 4-tuple of {sets of} covariance matrices.
Let 
$\Mu$ denote the family of such stationary policies.
 For any prior 4-tuple  $P_0$ (recall notation  (\ref{eq:4tuple}))
  and policy $\mu \in \Mu$ chosen by the micro-manager, define the stopping time $\tau = \inf\{k: u_k=1\}$.
 The following cost  is associated with the sequential decision procedure:
\beq \label{eq:csdef}
J_\mu(P) = \Ep\{ (\tau-1) \cop
+  \Cb(P_\tau)| P_0=P \}.
\eeq
Here $c_\nu$ is the radar operating cost and $\bar C$ the stopping cost introduced in Section~\ref{sec:formu}.
Also, $\Ep$ denotes expectation with respect to stopping time $\tau$ and initial condition $P$.
(A  measure-theoretic definition of $\Ep$, 
which involves an absorbing state to deal with stopping time $\tau$, is given in~\cite{HL96}).

The goal is to determine the optimal stopping time $\tau$ with minimal cost, that is, compute the optimal policy $\mu^* \in \Mu$ to minimize (\ref{eq:csdef}). Denote the optimal cost as
\beq \label{eq: optim}
J_{\mu^*}(P) = \inf_{\mu \in \Mu} J_\mu(P).
\eeq
 The existence of an optimal stationary policy $\mu^*$ follows from \cite[Prop.1.3, Chapter 3]{Ber00b}.
Since $\cop$ is  non-negative,  for the conditional entropy cost function of Case 4 in Section \ref{sec:formu}, 
 stopping is guaranteed
in finite time, i.e., $\tau$ is finite with probability~1. For Cases (1) to (3), in general $\tau$ is not necessarily finite --
however, this does not cause problems  from a practical point of view since the micro-manager  has typically
a pre-specified
 upper time bound  at
which it always chooses $u_k = 1$ and reverts back to the macro-manager. Alternatively,
for Cases (1) to (3), if one truncates
$\Cb(P)$ to some upper bound, then again stopping is guaranteed in finite time.

Considering  the above  cost (\ref{eq:csdef}),
the optimal stationary policy $\mu^* \in \Mu$ and associated value function
 $\Vb(P) = J_{\mu^*}(P)$
are the solution of the following
 ``Bellman's dynamic programming  equation'' \cite{HS84} (Recall our notation $ P = (P^a,\bP^a,P^{-a},\bP^{-a})$.)
\begin{align}
\label{eq:valuefcn}
   \Vb(P) &= \min\bigl\{ \Cb(P),
   \cop + \mathbb{E}_{z}\left[V\bigl(\mathcal{R}(P^{a},z^a),\mathcal{L}(\bP^a),\mathcal{R}(P^{-a},z^{-a}),
   \mathcal{L}(\bP^{-a})
   \bigr)\right]  \bigr\},\nonumber\\
 \mu^*(P) &= \argmin\bigl\{
 \Cb(P),
   \cop + \mathbb{E}_{z}\left[V\bigl(\mathcal{R}(P^{a},z^a),\mathcal{L}(\bP^a),\mathcal{R}(P^{-a},z^{-a}),
   \mathcal{L}(\bP^{-a})
   \bigr)\right]
 \bigr\}, \end{align}
where $\mathcal{R}$ and $\mathcal{L}$ were defined in (\ref{eq:riccEqn}) and  (\ref{eq:liapunov}).
  Here $\R(P^{-a},z^{-a})$ denotes the Kalman filter covariance  update for the $L-1$ lower priority targets according to
   (\ref{eq:riccEqn}).
  Our goal is to characterize the optimal policy $\mu^*$ and optimal  stopping set defined as
\beq \Rs = \{(P^a,\bP^a,P^{-a},\bP^{-a}): \mu^*(P^a,\bP^a,P^{-a},\bP^{-a})  = 1\}. \label{eq:stopset}\eeq
In the special Case 4 of Section \ref{sec:formu}, when $\alpha^l = 0$, then 
$\Rs = \{(P^a,P^{-a}): \mu^*(P^a,P^{-a})  = 1\}$.

The dynamic programming equation (\ref{eq:valuefcn}) does not translate into practical solution methodologies since the  space of
$P$, 4-tuples of {sets of} positive definite matrices, is uncountable, and it is not possible to compute the optimal decision policy in closed form.

\subsection{Main Result: Monotone  Optimal Decision Policy}
\label{sec:existence}
Our main result below  shows that the  optimal decision policy $\mu^*$ 
is a monotone function of  the covariance matrices of the targets.
To characterize $\mu^*$  in the sequential decision problem below,
we introduce the following notation:\\
Let  $m$ denote the dimension of the state
$s$ in (\ref{eq:gauss}). (In the GMTI radar example $m=4$). \\
Let $\pdf$ denote the set of all $m\times m$  real-valued, symmetric  positive semi-definite matrices.
For  $P, Q \in \pdf$ define the positive definite partial ordering $\gr $ as $P\gr Q$ if $x^T P x \geq x^T Q x$ for all $x\neq 0$, and $P\gs Q$ if $x^T P x > x^T Q x$ for $x\neq 0$.
Define $\lr$ with the inequalities reversed. Notice that $[\pdf,\gr]$ is a partially ordered set (poset).\\
Note that ordering positive definite matrices also orders  their eigenvalues. Let $x=(x_1,\ldots,x_m)$ and $y=(y_1,\ldots,y_m)$ denote vectors with elements in $\reals_+$.
Then 
define  the componentwise partial order on $\reals^m$ (denoted by  $\ol$) as
 $x \ol y$ (equivalently, $y \og x$) if
$x_i \leq y_i$ for all $i=1,\ldots,m$.

For any matrix $P \in \pdf$, let $\eig_P \in \reals_+^m$  denote the eigenvalues of  $P$ arranged in decreasing order as a vector.
Note $P \gr Q$ implies $\eig_P \og  \eig_Q$.  Clearly, $[\reals_+^m,\og]$ is a poset.

Define {scalar function} $f$ to be increasing\footnote{Throughout this paper, we use the term ``increasing" in the weak sense. That is ``increasing" means
non-decreasing. Similarly, the term ``decreasing" means non-increasing.} if $\eig_P \ol \eig_Q $ implies $f(\eig_P) 
\leq f(\eig_Q)$, or equivalently, if $P \lr Q$ implies $f(P) < f(Q)$. Finally we say that $f(P^{-a})$ is increasing in $P^{-a}$ if $f(\cdot)$  is increasing
in each component $P^l$ of $P^{-a}$, $l \neq a$.

The following is the main result of this paper regarding the policy $\mu^*(P^a,\bP^a,P^{-a},\bP^{-a})$.

\begin{framed}
\begin{theorem} \label{thm:main}
Consider the sequential detection problem (\ref{eq:csdef}) with stochastic observability
cost (\ref{eq:stop}) and stopping set (\ref{eq:stopset}).
\begin{enumerate}
\item The optimal decision  policy $\mu^*(P^a,\bP^a,P^{-a},\bP^{-a})$ 
is  increasing in $P^a$,
decreasing in $\bP^a$, decreasing in $P^{-a}$, and increasing in $\bP^{-a}$ 
on the poset $[\pdf,\gr]$. 
Alternatively, $\mu^*(P^a,\bP^a,P^{-a},\bP^{-a})$ is increasing in $\eig_{P^a}$,
decreasing in $\eig_{{\bP}^a}$,
 decreasing in $\eig_{P^{-a}}$ and increasing in $\eig_{\bP^{-a}}$
on the poset $[\reals_+^m,\og]$.  Here $\eig_{P^{-a}}$  denotes the $L-1$ vectors of eigenvalues $\eig_{P^l}$, $l\neq a$
(and similarly for  $\eig_{{\bP}^a}$).
\item  In the special case when $\alpha^l= 0$ for all $l\in \{1,\ldots,L\}$, (i.e., Case 4 in Section \ref{sec:formu} where
stopping cost is the conditional entropy) 
the optimal policy $\mu^*(P^a,P^{-a})$  is increasing in $P^a$
and decreasing in $P^{-a}$  on the poset $[\pdf,\gr]$. 
Alternatively, $\mu^*(P^a,P^{-a})$ is  increasing in $\eig_{P^a}$,
and 
 decreasing in $\eig_{P^{-a}}$ 
on the poset $[\reals_+^m,\og]$. \qed\end{enumerate} 
\end{theorem} \end{framed}

The proof is  in Appendix \ref{app:th1}.
 The monotone property of the  optimal decision policy $\mu^*$ 
is useful since (as described in Section \ref{sec:approx}) {parametrized monotone} policies are
 readily implementable at the radar micro-manager level and can be adapted in real time.
Note that in the context of GMTI radar, the above policy  is 
equivalent to the radar  micro-manager
 {\em opportunistically} deciding when to stop looking at a target: If the measured
 quality of the current target  is better than some threshold, then continue; otherwise stop.
 
 To get some intuition, consider the second claim of Theorem \ref{thm:main} when
 each state process has dimension $m=1$. Then the covariance of each target is a  non-negative scalar.
 The second claim of Theorem \ref{thm:main} says that there exists a threshold switching curve
$ P^a = g(P^{a})$, where $g(\cdot)$ is increasing in each element of $P^{a}$, such that for 
$P^a < g(P^{-a})$ 
it is optimal to stop,
and for $P^{a} \geq g(P^{-a})$ 
 it is optimal to continue.  This is illustrated in Figure \ref{fig:threshold}. Moreover, since
$g$ is monotone, it is 
differentiable almost everywhere (by Lebesgue's theorem).
 
 \begin{figure}[h]
\centerline{\includegraphics[scale=0.3]{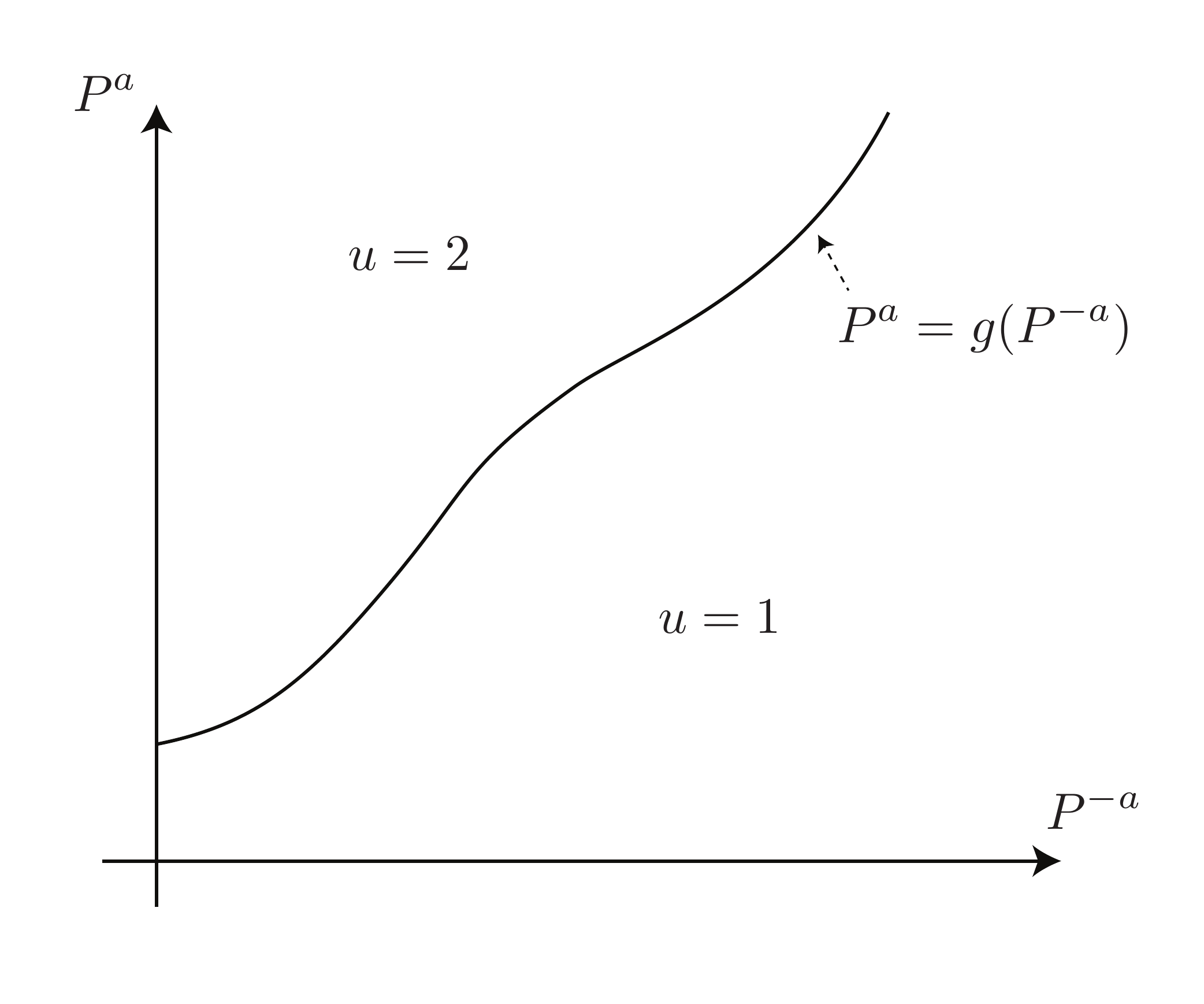}}
\caption{Threshold switching curve for optimal decision policy $\mu^*(P^a,P^{-a})$. 
Claim 2 of  Theorem \ref{thm:main} says that the optimal decision policy
is characterized by  a monotone
increasing threshold
curve $g(\cdot)$ when each target has state dimension $m=1$.}\label{fig:threshold}
\end{figure}

To prove Theorem \ref{thm:main} we will require  the following monotonicity  result regarding
the Riccati and Lyapunov equations of the Kalman covariance update. This  is proved in Appendix~\ref{sec:bitmead}. Below $\det(\cdot)$ denotes determinant.

\begin{framed}
\begin{theorem} \label{thm:bitmead}
Consider the Kalman filter Riccati covariance update,  $\R(P,z)$, defined in (\ref{eq:riccEqn}) with {possibly} missing measurements,  and Lyapunov  covariance update, $\L(P),$ defined in (\ref{eq:liapunov}).
 The  following properties hold
 for $P \in \pdf$ and $z\in \reals^{\mz}$ (where $\mz$ denotes the dimension of the observation
 vector $z$ in (\ref{eq:gauss}))
 :\\
(i)  $\frac{\text{det}(\L(P))}{\text{det}(P) }$ and (ii) $\frac{\text{det}(\R(P,z))}{\text{det}(P)}$  are
monotone decreasing in $P$
on the poset $[\pdf,\gr]$. \qed
\end{theorem} \end{framed}

\noindent{\em Discussion}:
An important property of Theorem  \ref{thm:bitmead} is that  stability of  the target system matrix $F$  (see (\ref{eq:state_dyn})) is not required.  In target tracking models (such as (\ref{eq:gauss})), $F$ has  eigenvalues at 1 and is therefore not stable.
By using  Theorem  \ref{thm:bitmead},   Lemma \ref{lem:bob} (in Appendix \ref{app:prelim}) shows that
the stopping cost involving stochastic observability is a monotone function of the covariances. 
 This monotone property of the stochastic observability of a Gaussian process is  of independent interest.

 Instead of stochastic observability (which deals with log-determinants), suppose we had chosen the stopping
cost in terms of the   trace of the covariance
matrices. Then,  in general, it is not true  that $\text{trace}(\R(P,z)) - \text{trace}(P)$ is decreasing in $P$ on the poset $[\pdf,\gr]$. Such a result typically requires stability of $F$. 

\section{Parametrized Monotone Policies and Stochastic Optimization Algorithms}
\label{sec:approx}

Theorem \ref{thm:main} shows that the optimal sequential decision policy $\mu^*(P) = \arginf_{\mu \in \Mu} J_\mu(P)$
  is monotone in $P$.
Below, we characterize and compute
optimal parametrized decision policies of the form $\mu_{\theta^*}(P) = \arginf_{\theta \in \Theta} J_{\mu_\theta}(P)$ 
for the sequential detection problem formulated in Section \ref{sec:seqdetf}.
Here $\theta \in \Theta$ denotes a suitably
chosen finite dimensional parameter and $\Theta$ is a subset of Euclidean space.
%
 Any such parametrized policy $\mu_{\theta^*}(P)$  needs to capture the essential feature
of Theorem \ref{thm:main}: it needs to be   decreasing in $P^{-a}, \bP^a$ and increasing in $P^a,\bP^{-a}$.
In this section, we derive several examples of parametrized policies that satisfy this property.
We then present simulation-based adaptive filtering (stochastic approximation) algorithms to estimate these optimal  parametrized policies.  To summarize,
 instead of attempting to solve an intractable dynamic programming problem (\ref{eq:valuefcn}), we 
 exploit the monotone structure of the optimal decision policy (Theorem \ref{thm:main}) to 
estimate a parametrized optimal monotone policy (Algorithm \ref{alg1} below).

\subsection{Parametrized Decision  Policies}
Below we give several examples of parametrized decision policies for the sequential
detection problem that are monotone in the covariances. Because
such parametrized policies satisfy the conclusion of Theorem~\ref{thm:main}, they can be used to approximate the  monotone optimal
policy of the sequential detection problem.  Lemma \ref{lem:suff} below shows that the constraints we specify are necessary and sufficient for the parametrized policy to be monotone implying that such policies $\mu_{\theta^*}(P)$
are an approximation to the 
optimal policy $\mu^*(P)$ within the appropriate parametrized class $\Theta$.

First we consider 3 examples of parametrized policies that are linear in the  vector of eigenvalues $\eig$ (defined in Section
\ref{sec:existence}).
Recall that $m$ denotes the dimension of state $s$ in
(\ref{eq:gauss}).
 Let 
$ \theta^l$ and $\bt^l \in \Theta = \reals_+^m$  denote the parameter vectors that parametrize the policy
$\mu_\theta$ defined as
\begin{align} \label{eq:maxpar}
\mu_\theta(\eig^a,\eig^{-a}) =  \begin{cases} 1 \text{ (stop), } &\hskip -3mm \text{if } 
- {\theta^{a}}^\p \eig_{P^a} +  {\bt^a}^\p \eig_{\bP^{a}} + 
\max_{l\neq a} 
\left[ 
{\theta^{l}}^\p \eig_{P^{l}}  -{\bt^l}^\p \eig_{\bP^{l}}
\right] \geq 1, \\
2 \text{ (continue), } & \hskip -3mm\text{otherwise.} 
\end{cases} \end{align}

\begin{align} \label{eq:minpar}
\mu_\theta(\eig^a,\eig^{-a}) =  \begin{cases} 1 \text{ (stop), } &\hskip -3mm \text{if } 
- {\theta^{a}}^\p \eig_{P^a} +  {\bt^a}^\p \eig_{\bP^{a}} + 
\min_{l\neq a} 
\left[ 
{\theta^{l}}^\p \eig_{P^{l}}  -{\bt^l}^\p \eig_{\bP^{l}}
\right] \geq 1, \\
2 \text{ (continue), } &\hskip -3mm \text{otherwise.} 
\end{cases} \end{align}

\begin{align} \label{eq:avgpar}
\mu_\theta(\eig^a,\eig^{-a}) =  \begin{cases} 1 \text{ (stop), } &\hskip -3mm \text{if } 
- {\theta^{a}}^\p \eig_{P^a} +  {\bt^a}^\p \eig_{\bP^{a}} + 
\sum_{l\neq a} 
\left[ 
{\theta^{l}}^\p \eig_{P^{l}}  -{\bt^l}^\p \eig_{\bP^{l}}
\right] \geq 1, \\
2 \text{ (continue), } &\hskip -3mm \text{otherwise.} 
\end{cases} \end{align}

As a fourth example, consider the parametrized policy in terms of covariance matrices. Below  $ \theta^l $
and $\bt^l \in \reals^m$ are unit-norm  vectors,
i.e, ${\theta^l}^\p \theta^l = 1$ and
${\bt^l}^\p \bt^l = 1$
for $l=1,\ldots,L$. Let $\mathcal{U}$ denote the space of  unit-norm vectors. Define the parametrized policy
$\mu_\theta$, $\theta \in \Theta = \mathcal{U}$ as
\begin{align} \label{eq:covpar}
\mu_{\theta}(P^{a},P^{-a})  = \begin{cases} 1 \text{ (stop), } &\hskip -6mm \text{ if }  - {\theta^a}^\p P^a \theta^a + 
{\bt^a}^\p \bP^a \bt^a + 
\sum_{l\neq a}  {\theta^l}^\p P^l \theta^l - {\bt^l}^\p \bP^l  \bt^l \geq 1,\\
2 \text{ (continue), } &\hskip -6mm \text{ otherwise. }  
\end{cases}
\end{align}

The following lemma states that the above parametrized policies satisfy the conclusion of Theorem \ref{thm:main}
that the policies are monotone.
 The proof is straightforward and hence omitted.

\begin{lemma} \label{lem:suff}
Consider each of the parametrized policies  (\ref{eq:maxpar}), (\ref{eq:minpar}), (\ref{eq:avgpar}).
Then $\theta^l,\bt^l \in \Theta = \reals_+^m$ is necessary and sufficient
for the parametrized policy $\mu_\theta$ to be monotone increasing in $P^a,\bP^{-a}$ and decreasing in $P^{-a},\bP^a$.
For (\ref{eq:covpar}), $\theta \in \Theta = \mathcal{U}$ (unit-norm vectors) is necessary and sufficient
for the parametrized policy $\mu_\theta$ to be monotone increasing in $P^a,\bP^{-a}$ and decreasing in $P^{-a},\bP^a$. \\ \mbox{} \qed
\end{lemma}

Lemma \ref{lem:suff} says  that since the constraints on the parameter vector $\theta$ are necessary and sufficient for a monotone policy, the classes of  policies
(\ref{eq:maxpar}), (\ref{eq:minpar}), (\ref{eq:avgpar}) and (\ref{eq:covpar})
 do not leave out any 
monotone policies; nor do they  include
any non monotone  policies.
 Therefore optimizing  over $\Theta$ for each case 
 yields the best   approximation to
the optimal policy within the appropriate class.

\noindent {\em Remark}: Another example of a parametrized policy that satisfies Lemma \ref{lem:suff} is obtained by replacing
$\eig_{X}$ with $\log \det(X)$  in
  (\ref{eq:maxpar}), (\ref{eq:minpar}), (\ref{eq:avgpar}). In this case, the parameters $\theta^a, \bt^a, \theta^l, \bt^l$ are scalars.
However, numerical studies (not presented here) show that this scalar parametrization is not rich enough to yield useful decision policies.

\subsection{Stochastic Approximation Algorithm to estimate $\theta^*$}
Having characterized monotone parameterized policies above,
our next goal is to compute 
 the optimal parametrized policy $\mu_{\theta^*}$ for the sequential detection problem described
 in Section \ref{sec:seqdetf}.
  This can
 be formulated as the following stochastic optimization problem:
\begin{align}
J_{\mu_{\theta^*}} &= \inf_{\theta \in \Theta}  J_\theta(P^a,\bP^a,P^{-a},\bP^{-a}), \nonumber
\\ 
\text{ where } 
J_\theta(P^a,\bP^a,P^{-a},\bP^{-a}) &= \Ept\{ (\tau-1) \cop
+  \Cb(P^a_\tau,\bP^a_\tau,P^{-a}_\tau,\bP^{-a}_\tau
\,| P_0=P,\bP_0 = \bP \}.\label{eq:simobj}
\end{align}
Recall that $\tau$ is the stopping time at which stop action $u=1$ is applied, i.e. $\tau = \inf\{k:u_k = 1\}$.
  
The optimal parameter $\theta^*$ in
(\ref{eq:simobj}) can be computed by simulation-based stochastic optimization algorithms as we now describe.  
Recall that for the first three examples above (namely, (\ref{eq:maxpar}), (\ref{eq:minpar})
and (\ref{eq:avgpar})), there is the explicit constraint that
  $\theta^l$ and $\bt^l \in \Theta = \reals^m_+$. This constraint  can be eliminated
straightforwardly by
choosing each component of $\theta^l$ as {$\theta^l(i)=\left[\phi^l(i)\right]^2$}
where $\phi^l(i) \in \reals$. 
The optimization problem (\ref{eq:simobj}) can then be formulated in
 terms of this new {unconstrained} parameter vector $\phi^l \in \reals^m$.

 In the fourth example above, namely (\ref{eq:covpar}), the parameter  $\theta^l$ is constrained
 to the boundary set of the $m$-dimensional unit hypersphere $\mathcal{U}$.
This constraint can be eliminated by parametrizing $\theta^l$ in terms of
 spherical coordinates $\phi$ as follows: Let    \beq
   \theta^l(1) = \cos \phi^l(1), \; \theta^l(i) = \prod_{j=1}^{i-1} \sin \phi^l(j) \cos \phi^l(i), \;
   i=2,\ldots, m-1, \quad \theta^l(m) = \prod_{j=1}^m \sin \phi^l(j). 
   \eeq
where $\phi^l(i) \in \reals $, $i=1,\ldots,m$ denote a parametrization of $\theta$. Then it is
trivially verified that {$\theta^l\in\mathcal U$}. Again the optimization problem (\ref{eq:simobj}) can then be formulated in
terms of this new unconstrained parameter vector $\phi^l\in \reals^m$.
   
\begin{algorithm}[h]
\caption{Policy Gradient Algorithm for computing optimal parametrized  policy} \label{alg1}
Step 1: Choose initial threshold coefficients  $\phi_{0}$ and parametrized policy
$\mu_{\theta_{0}}$.
\\
Step 2: For iterations $n=0,1,2,\ldots$ 
\begin{itemize}
\item Evaluate sample cost 
$
\hat{J}_n(\mu_{\phi}) = (\tau-1) \cop
+  \Cb(P^a_\tau,\bP^a_\tau, P^{-a}_\tau, \bP^{-a}_\tau)
$.\\
Compute gradient estimate $\nablat \hat{J}_n(\mu_{\phi})$ as:
$$ \nablat J_n = \frac{\displaystyle
J_n({\phi_n+\omega_n \d_n}) - J_n({\phi_n-\omega_n \d_n})}{\displaystyle 2
\omega_n } \d_n, \quad 
\d_n(i) = \begin{cases}
-1, & \text{ with probability } 0.5, \\
+1, & \text{with probability } 0.5.\end{cases}
 $$ 
 Here $\omega_n  = \frac{\omega}{{(n+1)}^\gamma}$ denotes the gradient step size with $0.5 \leq \gamma \leq 1$ and {$\omega>0$}.

\item Update threshold coefficients  $\phi_n$  via 
(where $\epsilon_n$ below denotes step size)
\beq \label{eq:sa}
\phi_{n+1} = \phi_n - \epsilon_{n+1} \nablat  \hat{J}_n(\mu_{\phi}), \quad
\epsilon_n = \epsilon/(n+1+s)^\zeta, \quad 0.5 < \zeta \leq  1, \;
\text{ and }\epsilon , s > 0. \eeq
\end{itemize} \label{alg:spsa}
\end{algorithm}

  Several possible simulation based stochastic approximation algorithms can be used to estimate $\mu_{\theta^*}$ in (\ref{eq:simobj}). In our numerical examples, we used Algorithm \ref{alg1}  to estimate the optimal parametrized policy.   Algorithm \ref{alg1} is a 
 Simultaneous Perturbation Stochastic Approximation (SPSA) algorithm
\cite{Spa03};  see \cite{Pfl96} for other more sophisticated
gradient estimators. Algorithm \ref{alg1}
generates a sequence of estimates $\phi_{n}$ and thus $\theta_n$,
 $n=1,2,\ldots,$ that
converges to a local minimum $\theta^*$  of (\ref{eq:simobj}) with
policy $ \mu_{\theta^*}(P) $. In Algorithm \ref{alg1} we denote the policy as $\mu_\phi$ since $\theta$ is 
parametrized in terms of $\phi$ as described above.

The SPSA algorithm \cite{Spa03} 
picks a single
random direction $\d_n$ (see Step 2) along which  the derivative is
evaluated {after} each batch $n$.  As is apparent  from Step 2 of
Algorithm \ref{alg1}, 
 evaluation of the gradient estimate $\nablat J_n$ 
requires only 2 batch simulations. This is unlike the well known Kiefer-Wolfowitz stochastic approximation
algorithm \cite{Spa03} where $2 m$ batch simulations are required to evaluate the gradient
estimate.
Since  the stochastic gradient algorithm
(\ref{eq:sa}) converges to  a local optimum, it is necessary
to {retry with} several {distinct} initial conditions. 


\section{Application: GMTI Radar Scheduling and Numerical Results} \label{sec:numerical}

This section  illustrates the  performance of the monotone  parametrized policy (\ref{eq:simobj}) computed via Algorithm \ref{alg1} in
a GMTI radar scheduling  problem.
We first  show that the nonlinear measurement model of a GMTI tracker
can be approximated satisfactorily by the linear Gaussian model (\ref{eq:gauss})
that was used above. Therefore the main result Theorem \ref{thm:main} applies, implying that the optimal
radar micro-management
decision policy is monotone.
To 
illustrate these micro-management policies numerically, we then consider two important GMTI surveillance problems -- the target fly-by problem and the persistent surveillance problem.

\subsection{GMTI Kinematic Model and Justification of Linearized Model (\ref{eq:gauss})} \label{sec:appgmti}

 The observation model below is an abstraction based on approximating
several underlying pre-processing steps. For example,  given raw GMTI measurements,   space-time adaptive
processing (STAP) (which is a two-dimensional adaptive filter)  is used for near real-time detection,
see 
\cite{BDW10} and references therein. Similar observation models can be used as abstractions of 
 synthetic aperture radar (SAR) based processing.

 A modern GMTI radar manager operates on three time-scales (The description below is a 
{simplified variant} of an actual radar system.):
\begin{itemize}
\item Individual observations of target $l$ are obtained on the fast time-scale $t=1,2,\ldots$. The period at which  $t$ ticks is 
 typically $1$ milli-second. At this time-scale, ground targets can be considered to be static.
\item Decision epoch $k=1,2,\ldots,\tau$ is the time-scale at which the micro-manager and target tracker operate. 
Recall $\tau$ is the stopping time at which the  micro-manager decides to stop
and return control to the macro-manager.
The clock-period at which $k$ ticks is typically  $T= 0.1$  seconds. 
 At this epoch time-scale $k$, the targets move according to the kinematic model (\ref{eq:state_dyn}), (\ref{eq:obs}) 
below.  Each
epoch $k$ {is} comprised of intervals $t=1,2,\ldots, \Delta$ of the fast time-scale, where $\Delta$ is typically of the order of 100. So,
100 observations are integrated at the $t$-time-scale to yield a single observation at the $k$-time-scale.


\item The scheduling interval $n=1,2\ldots,$ is the time-scale at which the macro-manager operates.
Each scheduling interval $n$ is comprised  of $\tau_n$ decision epochs. This stopping time $\tau_n$ is determined by the micro-manager.
$\tau_n$ is typically  in the range 10 to  50  -- in absolute time it corresponds to the range 1 to 5  seconds. In such a time period, a ground target
moving at 50 km per hour moves approximately in the range 14 to 70 meters.
\end{itemize}

\subsubsection{GMTI Kinematic Model} \label{sec:gmtikinematic}
 The tracker assumes that each target $l \in \{1,\ldots,L\}$ has kinematic model  and GMTI observations \cite{BP99},
\begin{align}
\label{eq:state_dyn} 
	s^{l}_{k+1} & = Fs^{l}_{k} + Gw^{l}_{k},\\
	z^{l}_{k} & = \begin{cases} h(s^{l}_{k},\plat_k) + \frac{1}{\sqrt{\price^l \Delta}}v^{l}_{k},  &  \text{ with probability $p_d^l$,}\\
						\emptyset, &  \text{ with probability $1- p_d^l$.}  \end{cases}
	\label{eq:obs}
\end{align}
Here $z^{l}_{k}$ denotes a 3-dimensional 
(range, bearing and range rate)
observation vector of target $l$ at epoch time $k$ and $\plat_k$ denotes the Cartesian coordinates and speed of the platform (aircraft)
on which the GMTI radar is mounted.
 The noise processes
$w^{l}_{k}$ and $v^{l}_{k}/\sqrt{\price^l \Delta}$ are zero-mean Gaussian random vectors with covariance matrices $ Q^{l} $ and $ R^{l}(\price^l) $, respectively. 
 The observation $z_k$ in  decision epoch $k$ is the average of the $\price^l\Delta$ measurements obtained at the fast 
 time scale $t$.
Thus the observation noise variance in (\ref{eq:obs}) is scaled by the reciprocal of the target
priority $\price^l \Delta$.
In (\ref{eq:state_dyn}), (\ref{eq:obs}) for a GMTI system,
 \begin{align}
\label{eq:system_par} 
 F &= \left[ \begin{array}{cccc}
1 & T & 0 & 0 \\
0 & 1 & 0 & 0 \\
0 & 0 & 1 & T \\
0 & 0 & 0 & 1 \end{array} \right], \quad
G = \begin{bmatrix}  T^2/2 & 0 \\ T & 0 \\ 0 & T^2/2 \\ 0 & T \end{bmatrix},\quad
R(\price^l) = \frac{1}{{\price^l \Delta}} \left[ \begin{array}{ccc}
\sigma_{r}^{2} & 0 & 0 \\
0 & \sigma_{a}^{2} & 0 \\
0 &  0 & \sigma_{\dot{r}}^{2} \end{array} \right],\\
Q &= \left[ \begin{array}{cccc}
\frac{1}{4}T^{4}\sigma_{x}^{2} & \frac{1}{2}T^{3}\sigma_{x}^{2} & 0 & 0 \\
\frac{1}{2}T^{3}\sigma_{x}^{2} & T^{2}\sigma_{x}^{2} & 0 & 0 \\
0 & 0 & \frac{1}{4}T^{4}\sigma_{y}^{2} & \frac{1}{2}T^{3}\sigma_{y}^{2} \\
0 & 0 & \frac{1}{2}T^{3}\sigma_{y}^{2} & T^{2}\sigma_{y}^{2} \end{array} \right],\quad
h(s,\plat) = \left[ \begin{array}{c}
\sqrt{(x-\plat_x)^2+(y-\plat_y)^2+\plat_z^2} \\
\arctan{(\frac{y-\plat_y}{x-\plat_x})} \\
\frac{(x-\plat_x)(\dot{x}-\dot{\plat}_x)+(y-\plat_y)(\dot{y}-\dot{\plat}_y)}{\sqrt{(x-\plat_x)^2+(y-\plat_y)^2+ \plat_z^2}} \end{array} \right]. \nonumber
\end{align}
Recall that $T$ is typically 0.1 seconds.
The elements of $h(s,\plat)$ correspond to
 range, azimuth, and range rate, respectively. Also $\plat = (\plat_x,\dot{\plat}_x, \plat_y,\dot{\plat}_y)$ denotes the $x$ and $y$  position and speeds, respectively, and $\plat_z$ denotes the altitude, assumed to be constant, of the aircraft on which the GMTI radar is mounted.

\subsubsection{Approximation by Linear Gaussian State Space Model}
Starting with the nonlinear state space model (\ref{eq:state_dyn}), the aim below is to justify the use of the linearized model (\ref{eq:gauss}).
We start with  linearizing the  model (\ref{eq:state_dyn})  as follows; see \cite[Chapter 8.3]{Jaz70}. For each target $l$, consider a nominal deterministic
target  trajectory $\bar{s}_k^l$ and nominal measurement $\bar{z}_k^l$ where
$ \bs_{k+1}^l = F \bs_k^l$ , 
$\bar{z}_k^l = h(\bs_k^l,\plat)$.
Defining $\tilde{s}_k^l = s_k ^l- \bs_k^l$ and $\tilde{z}_k^l = z_k^l - \bar{z}_k^l$, a first order Taylor series expansion  around this nominal trajectory yields,
\begin{align}
\tilde{s}_{k+1}^l  &= F \tilde{s}_k^l + G w_k^l, \nonumber\\
\tilde{z}_{k}^l &= \begin{cases} \nabla_s h(\bs_k^l,\plat_k) \tilde{s}_k^l + R_1(s_k,\bs_k^l,\plat_k) + \frac{1}
{\sqrt{{\price^l \Delta}}}  v_k^l ,   &  \text{ with probability $p_d^l$,}\\
\emptyset, &  \text{ with probability $1- p_d^l$,}  \end{cases} \label{eq:heq}
\end{align}
where $\|R_1(s^l,\bs^l,\plat) \|  \leq  \frac{1}{2}\|(s^l-\bs^l)^\p \nabla^{2}h(\zeta,\plat)(s^l-\bs^l)\| $ and $\zeta = \gamma s^l + (1-\gamma) \bs^l $ for some $\gamma\in [0,1]$.
In the above equation,  $\nabla_s h(s,\plat) $ is the Jacobian matrix defined as (for simplicity we omit the superscript $l$ for target $l$),
\begin{align}
 \nabla_s h(s,\plat)_{ij} &
= \left[ \begin{array}{cccc}
\frac{\xr}{r} & 0 & \frac{\yr}{r} & 0 \\
\frac{-\yr}{\xr^2+\yr^2} & 0 & \frac{\xr}{\xr^2+\yr^2} & 0 \\
\frac{\dot{\xr}}{r} - \frac{\xr\yr\dot{\yr} + \xr^2\dot{\xr}}{r^{3}} & \frac{\xr}{r} & \frac{
\dot{\yr}}{r} - \frac{\xr\yr\dot{\xr} + \yr^2\dot{\yr}}{r^{3}}  & \frac{\yr}{r} \end{array} \right]
,\quad r = \sqrt{\xr^2+\yr^2+ \zr^2}. \label{eq:jacobian}
\end{align}
where
$\xr = x - \plat_x$, $\yr = y - \plat_y$ denotes the relative position of the target with respect to the platform
and $\dot{\xr}$, $\dot{\yr}$ denote the relative velocities. Since the target is ground based and the platform
is constant altitude, $\plat_z$ is a constant.

In (\ref{eq:heq}), $\nabla^2 h(\cdot,\cdot)$ denotes the $3\times 4\times 4$ Hessian tensor. By evaluating this Hessian tensor
for typical operating modes  and 
$k \leq 50$,  we show below  that
\beq \frac{\|R_1(s^l,\bs^l,\plat) \|} {\| \nabla_s h(\bs^l,\plat) \tilde s^l\|}  \leq 0.02, \quad  
 \frac{\|\nabla_s h(\bs_k^l,\plat_k)  - \nabla_s h(\bs_0^l,\plat_0)\|} {   \|\nabla_s h(\bs_k^l,\plat_k)\|} \leq 0.06. \label{eq:dedef}
\eeq
The first inequality above says that the model is approximately linear in the sense that the 
ratio of linearization error $R_1(\cdot)$ to linear term is small; the second inequality says that the model is approximately time-invariant, in the sense that the relative magnitude of the error between linearizing around $\bs_0$ and $\bs_k$
is small.
Therefore, on the micro-manager time scale,  the target dynamics can be viewed as a linear time invariant  state space model
(\ref{eq:gauss}). 

{\em Justification of (\ref{eq:dedef})}:
Using typical GMTI operating parameters, we  evaluate the bounds in (\ref{eq:dedef}).
Denote the state of the platform (aircraft) on which the radar is situated as
$
p_0=[p_{x,0},\dot{p}_x,p_{y,0},\dot{p}_y] = [-35000\text{m},100\text{m/s},-15000\text{m},20\text{m/s}]
$.
Then the platform height is
$
p_z = \sqrt{p_{x,0}^2+p_{y,0}^2}\tan{\theta_d}$,
where $\theta_d$ is the depression angle, typically between $10^\circ$ to $25^\circ$. We  assume a depression angle of $\theta_d = 15^\circ$ below yielding 
$p_z
= 10203.2 \text{m}$.
Next, consider typical behaviour of ground targets with
speed  15m/s (54 km/h) and select the following significantly different initial target state vectors
(denoted by superscripts $a-e$)
\begin{align}
s_0^a &= \begin{bmatrix}
100 & 3 & 40 & 7 \end{bmatrix},\;
s_0^b = \begin{bmatrix}
-20 & -4 & 200 & 1 \end{bmatrix},\;
s_0^c = \begin{bmatrix}
50 & 2 & 95 & 10 \end{bmatrix},\\
s_0^d &= \left[ \begin{array}{cccc}
-70 & 5 & -50 & -6 \end{array} \right],\;
s_0^e = \begin{bmatrix}
150 & -15 & 10 & 0 \end{bmatrix}.
\end{align}
Now, propagate these initial states using the  target model with $ T = 0.1$s, $\sigma_x = \sigma_y = 0.5$, $\sigma_r = 20$m, $\sigma_{\dot{r}} = 5$m/s, $\sigma_a = 0.5^\circ$ with a true track variability parameter $\sigma_p = 1.5$ (used for true track simulation as $\sigma_x$ and $\sigma_y$). 
Define (see (\ref{eq:dedef}) and recall that $\zeta = \gamma s + (1-\gamma) \bs $ for some $\gamma\in [0,1]$)
$$
D(\bs_k,\plat_k) \equiv \frac{\|\nabla_s h(\bs_k,\plat_k)  - \nabla_s h(\bs_0,\plat_0)\|} {   \|\nabla_s h(\bs_k,\plat_k)\|},
\quad E(s,\bs,\plat,\gamma) =\frac{ \frac{1}{2}\|(s-\bs)^\p D^{2}h(\zeta,\plat)(s-\bs)\| }
{\| \nabla_s h(\bs,\plat) (s - \bs)\|} .
$$
Tables \ref{tab1} to \ref{tab3} show how $D(\cdot)$ and $E(\cdot)$  evolve with iteration $k=10, 50, 100$.
The entries in the tables are small, thereby justifying
the linear time invariant state space model (\ref{eq:gauss}).

{\em Remark}: Since a linear Gaussian model is an accurate approximate model, most real GMTI trackers
use an extended Kalman filter. Approximate nonlinear filtering methods such as 
sequential Markov Chain Monte-Carlo methods (particle filters) are not
required. 

\begin{table}[p]
\centering
\begin{tabular}{|c||c|c|c|}\hline
 & $D(\bs_{10},\plat_{10})$ & $D(\bs_{50},\plat_{50})$ & $D(\bs_{100},\plat_{100})$ \\ \hline
${s}_0^a$ & 0.0010 & 0.0052 & 0.0104 \\ \hline
${s}_0^b$ & 0.0009 & 0.0049 & 0.0104 \\ \hline
${s}_0^c$ & 0.0010 & 0.0059 & 0.0119 \\ \hline
${s}_0^d$ & 0.0007 & 0.0040 & 0.0080 \\ \hline
${s}_0^e$ & 0.0010 & 0.0053 & 0.0112 \\ \hline
\end{tabular}\caption{Rate of change of Jacobian for various running times.}\label{tab1}
\end{table}
\begin{table}[p]
\centering
\begin{tabular}{|c||c|c|c|}\hline
 & $E(s_{10},\bs_{10},\plat_{10},0.1)$ & $E(s_{50},\bs_{50},\plat_{50},0.1)$& $E(s_{100},\bs_{100},\plat_{100},0.1)$ \\ \hline
${s}_0^a$ & 0.00019091 & 0.0010597 & 0.01395 \\ \hline 
${s}_0^b$ & 0.00020866 & 0.0011699 & 0.014375 \\ \hline 
${s}_0^c$ & 0.00019165 & 0.0010813 & 0.01453 \\ \hline 
${s}_0^d$ & 0.0002008 & 0.0011065 & 0.011844 \\ \hline 
${s}_0^e$ & 0.0002294 & 0.0012735 & 0.015946 \\ \hline
\end{tabular}\caption{Ratio of second-order to first-order term of Taylor series expansion for $\alpha = 0.1$.} \label{tab2}
\end{table}

\begin{table}[p]
\centering
\begin{tabular}{|c||c|c|c|}\hline
 & $E(s_{10},\bs_{10},\plat_{10},0.8)$ & $E(s_{50},\bs_{50},\plat_{50},0.8)$& $E(s_{100},\bs_{100},\plat_{100},0.8)$ \\ \hline
${s}_0^a$ & 0.00019267 & 0.0011104 & 0.014211 \\ \hline 
${s}_0^b$ & 0.00021104 & 0.0012164 & 0.014633 \\ \hline 
${s}_0^c$ & 0.00019447 & 0.0011228 & 0.014838 \\ \hline 
${s}_0^d$ & 0.00020148 & 0.0011386 & 0.012178 \\ \hline 
${s}_0^e$ & 0.00023083 & 0.0013596 & 0.016603 \\ \hline 
\end{tabular}\caption{Ratio of second-order to first-order term of Taylor series expansion for $\alpha = 0.8$.} \label{tab3}
\end{table}

\subsection{Numerical Example 1: Target Fly-by}\label{sec:flyby}
With the above justification of the model (\ref{eq:gauss}), 
we present the first numerical example.
Consider $L=4$ ground targets that are tracked by a GMTI  platform, as illustrated in Figure \ref{fig:sim}. The nominal range from the GMTI sensor to the target region is approximately $\tilde{r} = 30$km.
For this example, the initial (at the start of the micro-manager cycle) estimated and true target states of the four targets are given
in Table \ref{tab:init}.
\begin{table}[p]
\begin{center}
\begin{tabular}{ll}
$\hat{s}_0^1 = [130,5.5,84,8.1]^\mathrm{T}$, & $s_0^1 = [100,3,40,7]^\mathrm{T}$\\
$\hat{s}_0^2 = [-47.88,-2.38,210.41,0.418]^\mathrm{T}$, & $s_0^2 = [-20,-4,200,1]^\mathrm{T}$\\
$\hat{s}_0^3 = [55.84,2.37,121.74,9.56]^\mathrm{T}$, & $s_0^3 = [50,2,95,10]^\mathrm{T}$\\
$\hat{s}_0^4 = [-55.13,5.75,-68.41,-6.10]^\mathrm{T}$, & $s_0^4 = [-70,5,-50,-6]^\mathrm{T}$
\end{tabular}
\end{center}
\caption{Initial Target States and Estimates.} \label{tab:init}
\label{default}
\end{table}

\begin{figure}[p]
\centerline{\includegraphics[scale=0.4]{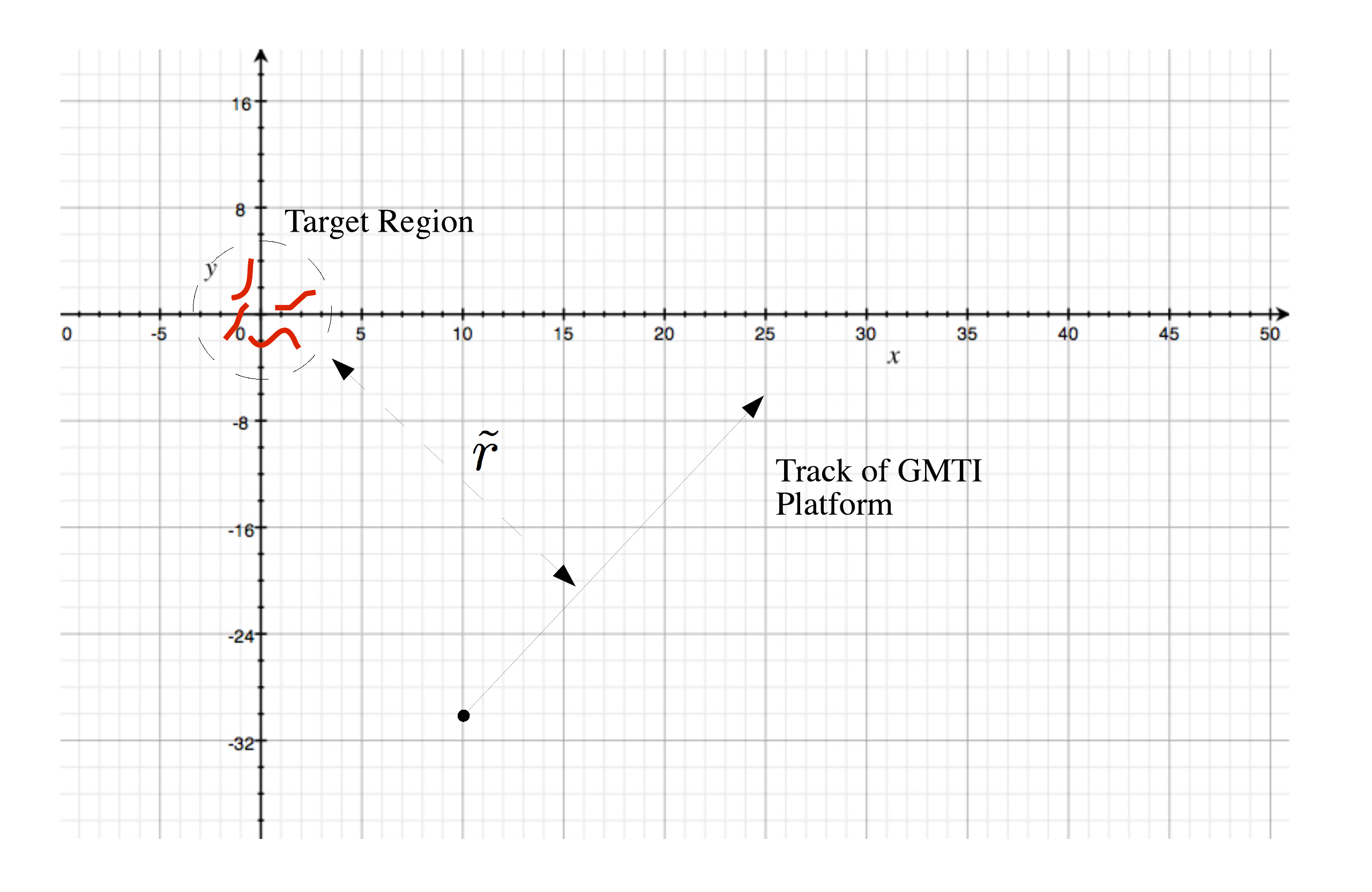}}
\caption{Target Fly-by Scenario. The GMTI platform (aircraft) moves with constant altitude and
velocity at nominal range $r=30$ km from the target region. ($r$ is defined in (\ref{eq:jacobian})).
Initial states of the four targets are specified in Table \ref{default}.}\label{fig:sim}
\end{figure}

We assume in this example that the most uncertain target is regarded as being the highest priority. Based on the initial states and estimates in Table \ref{default}, the mean square error values are, MSE$(\hat{s}_0^1) = 710.87$, MSE$(\hat{s}_0^2) = 222.16$, MSE$(\hat{s}_0^3) = 187.37$, and MSE$(\hat{s}_0^4) = 140.15$. Thus, target $l=1$ is the most uncertain and allocated the highest priority. So we denote $a=1$.

The simulation parameters are as follows: sampling time $T=0.1$s (see 
Section \ref{sec:gmtikinematic}); probability of detection $p_d = 0.75$ (for all targets, so superscript $l$ is omitted); track
standard deviations of target model $\sigma_x = \sigma_y = 0.5$m; measurement noise standard
deviations $\sigma_r = 20$m, $\sigma_a = 0.5^\circ$, $\sigma_{\dot{r}} = 5$m/s; and platform states $[p_x,\dot{p}_x,p_y,\dot{p}_y] = [10\text{km},53\text{m/s},-30\text{km},85\text{m/s}]$. We assume a target priority vector of $\price = [\price^1,\price^2,
\price^3,\price^4] = [0.6, 0.39, 0.008, 0.002]$.
Recall from (\ref{eq:obs})
 that the target priority scales the inverse of the covariance of the observation noise. 
 We chose an operating cost of $\cop = 0.8$, and the stopping cost of $\bar{C}(P^{-a})$
 specified in  (\ref{eq:avgp}), with constants $\alpha^{1,\ldots,4} = \beta^{2,\ldots,4} = 0.05$, $\beta^1 = 5$. The parametrized policy chosen for this example was $\mu_{\theta}(P^a,P^{-a})$ defined in (\ref{eq:covpar}).
We used the SPSA algorithm (Algorithm \ref{alg1}) 
to estimate the  parameter $\theta^*$ that optimizes the objective (\ref{eq:simobj}).
Since the SPSA converges to a local minimum, several initial conditions were evaluated via a  random search.

Figure \ref{fig:cost_vs_pd_cop} explores the sensitivity of the sample-path cost (achieved by the parametrized policy) 
with respect to probability of detection, $p_d$, and the operating cost, $\cop$. The sample-path cost increases with $\cop$ and decreases with $p_d$. Larger values of the operating cost, $\cop$, cause the radar micro-manager to specify the ``stop'' action sooner than for lower values of $\cop$. As can be seen in the figure, neither the sample-path cost or the average stopping time is particularly sensitive to changes in the probability of detection. However, as expected, varying the operating cost has a large effect on both the sample-path cost and the associated average stopping time.

\begin{figure}[p]
\centerline{\includegraphics[scale=0.45]{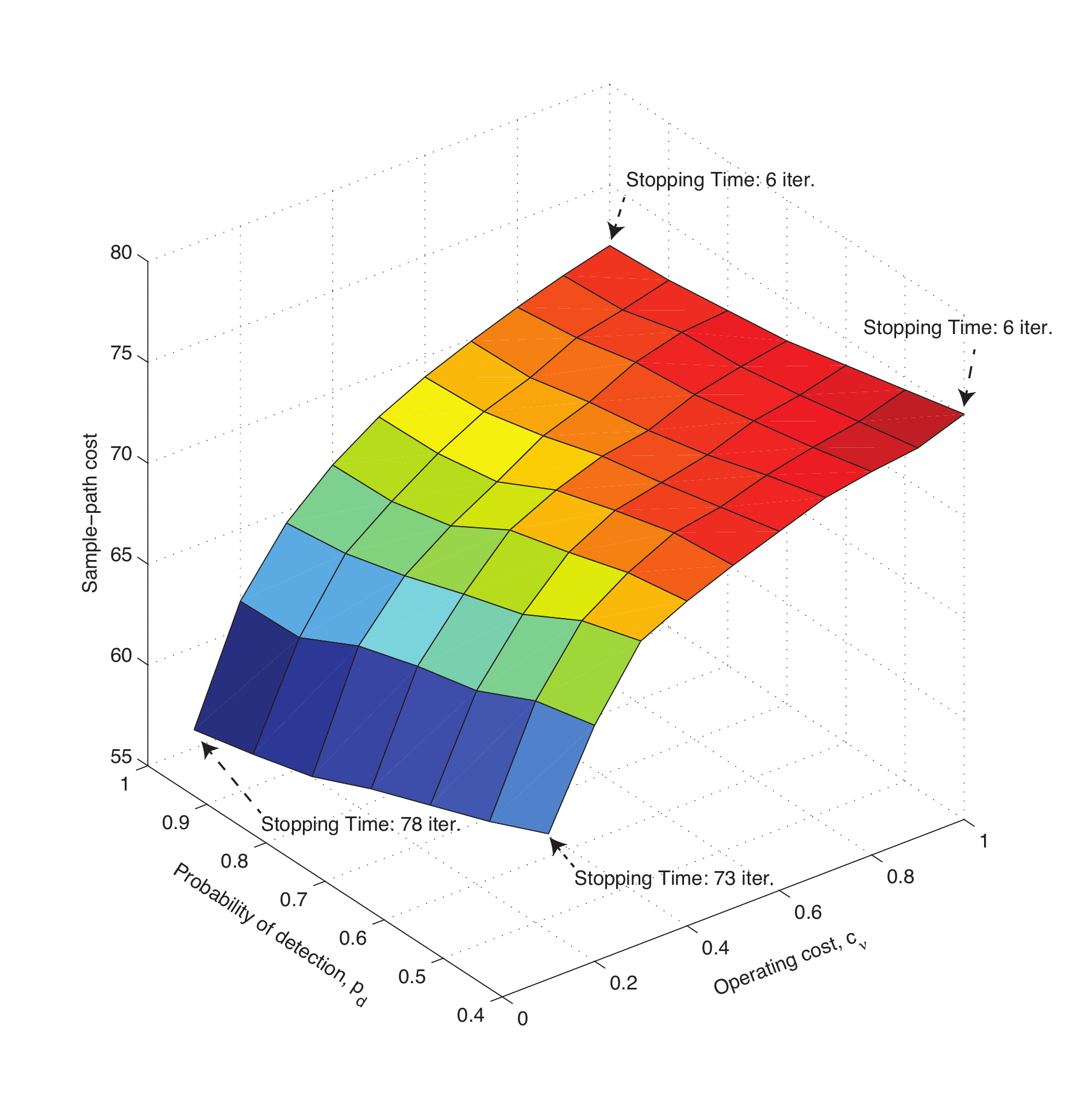}}
\caption{Dependence of the sample-path cost achieved  by the parametrized policy on the probability of detection, $p_d$, and the operating cost, $\cop$. The sample-path cost increases with the operating cost, but decreases with the probability of detection. Note the stopping times associated with the labelled vertices above.}\label{fig:cost_vs_pd_cop}
\end{figure}

Figure \ref{fig:cost_vs_IC} compares the optimal parametrized policy with periodic myopic policies. Such periodic myopic  policies 
stop at a deterministic pre-specified time (without considering state information) and then return control  to the macro-manager. 
The performance of the optimal parametrized policy is measured using multiple initial conditions. 
 As seen in Figure \ref{fig:cost_vs_IC}, the optimal parametrized policy is the lower envelope of all possible periodic stopping times, for each initial condition.
The optimal periodic policy is highly dependent upon the initial condition. The main performance advantage of the optimal parametrized policy is that it achieves virtually the same cost as the optimal periodic policy for any initial condition.

\begin{figure}[t]
\mbox{\subfigure[Sample-path cost]
{\includegraphics[scale=0.4]{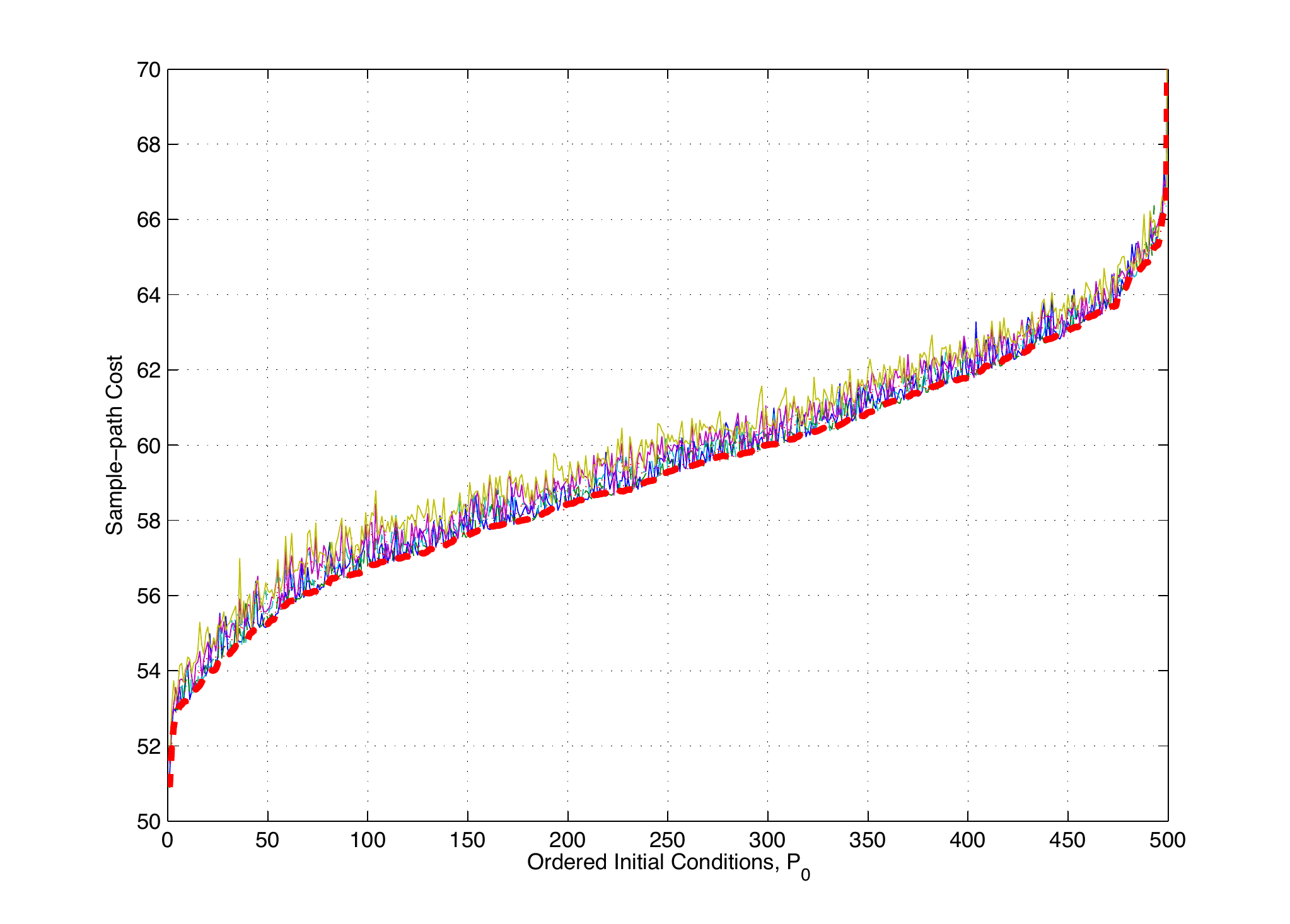}} \quad
{\subfigure[Magnified region]{\includegraphics[scale=0.4]{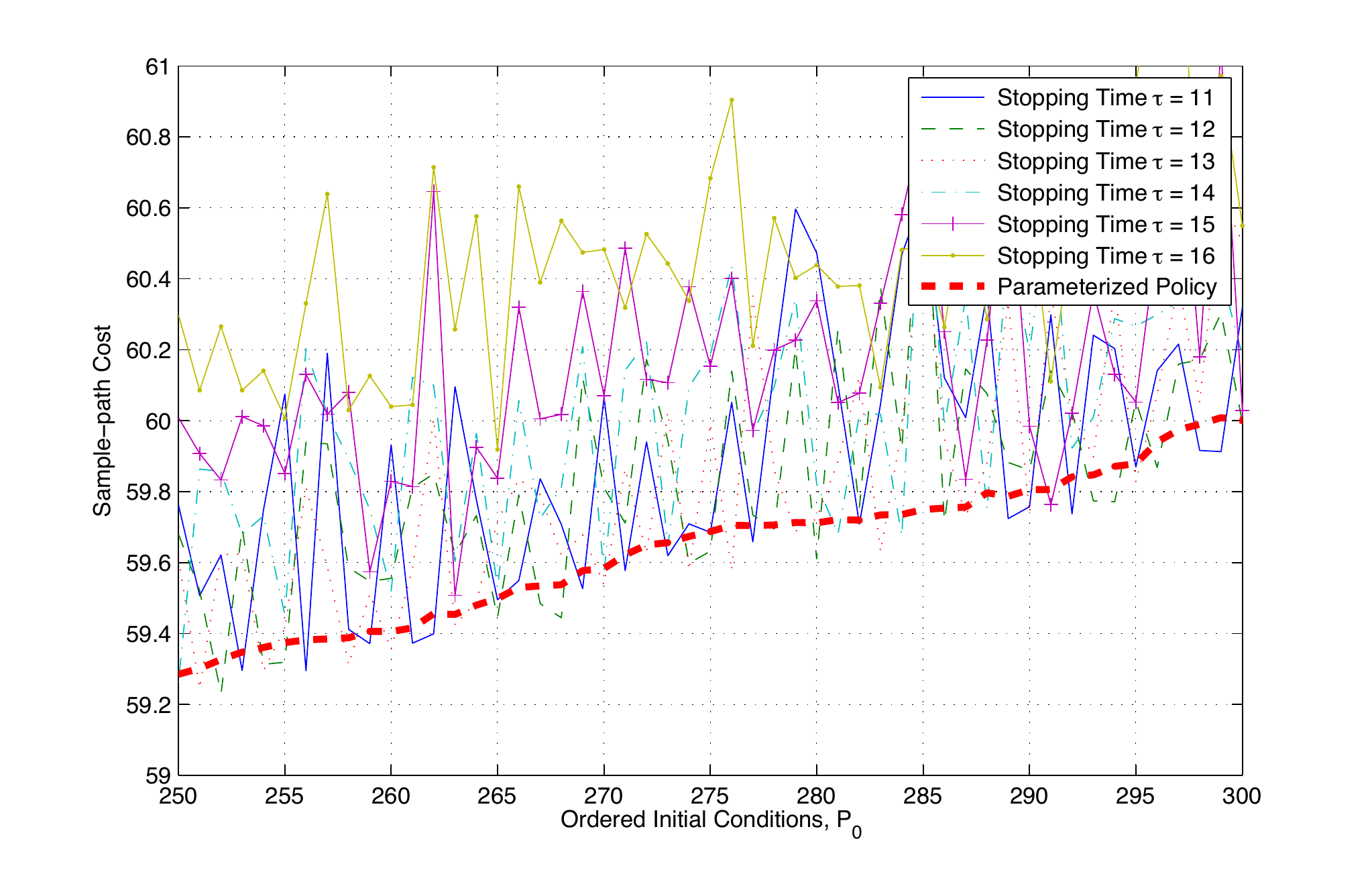}}}}
\caption{Plot of sample-path cost of  periodic policies and the parametrized policy (thick-dashed line) versus 
 initial conditions. These initial conditions are ordered with respect to the cost achieved using the 
 parametrized policy for that particular initial condition. Notice that the sample-path cost is the lower envelope of all deterministic stopping times for any initial condition.}\label{fig:cost_vs_IC}\end{figure}


\subsection{Numerical Example 2: Persistent Surveillance}
As mentioned in Section \ref{sec:intro},
persistent surveillance involves exhaustive surveillance of a region over long time intervals, typically over the period of several hours or weeks \cite{RHL07} and is useful in  providing critical, long-term battlefield information.
 Figure \ref{fig:persist} illustrates the persistent surveillance setup.
Here $\tilde{r}$ is the nominal range from the target region to the GMTI platform, assumed in our simulations to be approximately 30km. The points on the GMTI {platform} track labeled (1) -- (72) correspond to locations\footnote{The platform state at location $n\in\{1,2,...,72\}$ is defined as $p = [\tilde{r}\cos (n\cdot 5^\circ),-\tilde{v}\sin (n\cdot 5^\circ),\tilde{r}\sin (n\cdot 5^\circ),\tilde{v}\cos (n\cdot 5^\circ)]$} where we evaluate the Jacobian (\ref{eq:jacobian}).  Assume a constant platform orbit speed of 250m/s (or approximately 900km/h \cite{USAF2007}) and a constant altitude of approximately 5000m. Assuming 72 divisions along the 30km radius orbit, the platform sensor takes 10.4 seconds to travel between the track segments.  Using a similar analysis to the Appendix, the measurement model changes less than 5\% in $l_2$-norm in 10.4s,  thus the optimal parameter vector is approximately constant on each track segment.

\begin{figure}[p]
\centerline{\includegraphics[scale=0.3]{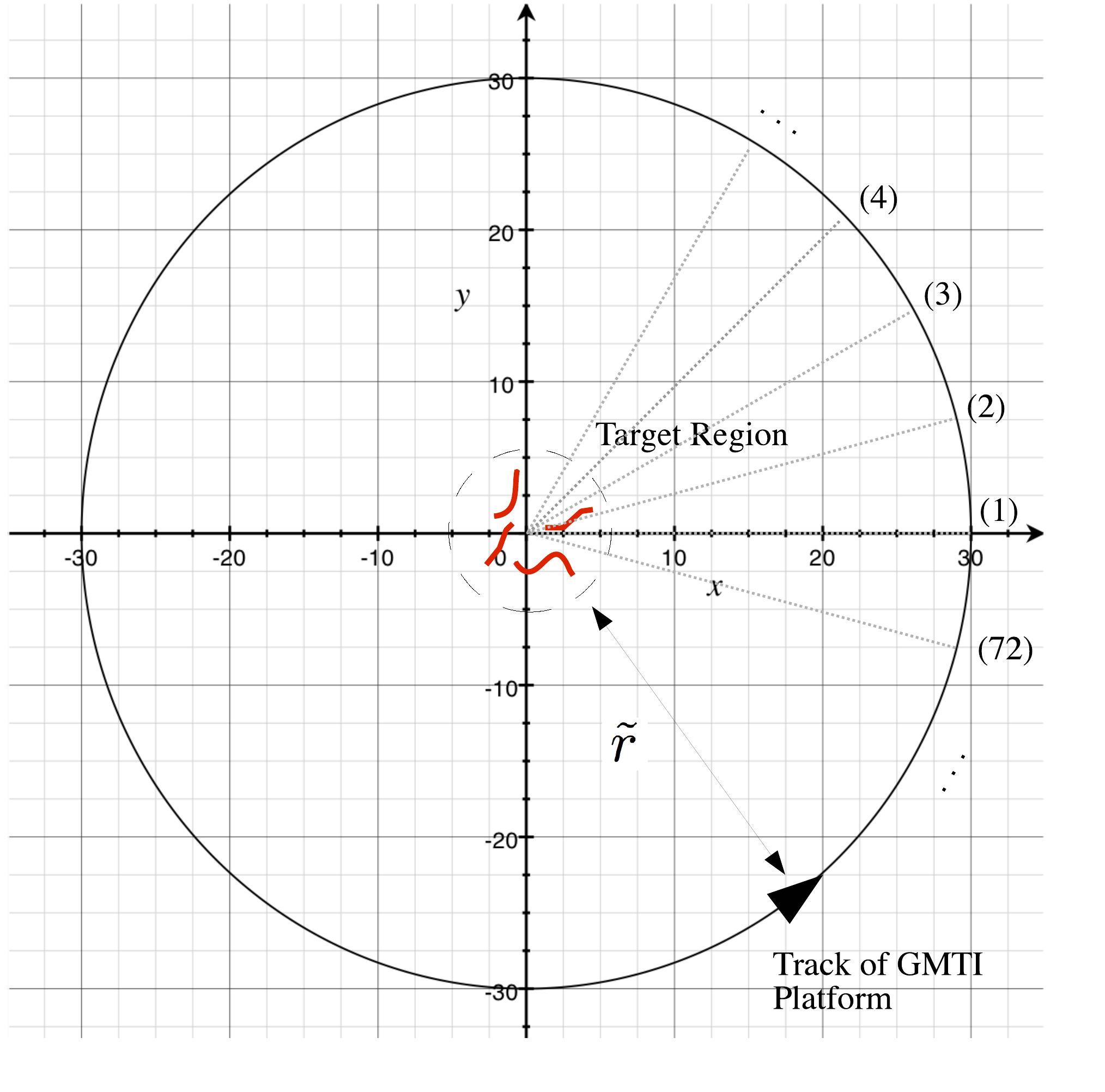}}
\caption{Representation of the persistent surveillance scenario in GMTI systems. The GMTI platform (aircraft) orbits the target region in order to obtain persistent measurements as long as targets remain within the target region. The nominal range from the platform to the target region is assumed to be 30km.}\label{fig:persist}
\end{figure}

Simulation parameters for this example are as follows: number of targets $L = 4$; sampling time $T=0.1$s; probability of detection $p_d = 0.9$; track variances of target model $\sigma_x = \sigma_y = 0.5$m; and measurement noise parameters $\sigma_r = 20$m, $\sigma_a = 0.5^\circ$, $\sigma_{\dot{r}} = 5$m/s. The platform velocity is now changing (assume a constant speed of $\tilde{v}=250$m/s), unlike  the previous example, which
assumed a constant velocity platform. Since the linearized model will be  different at each of the pre-specified points, (1)--(72), along the GMTI track, we  computed the optimal parametrized policy at each of the respective locations. The radar
manager then switches between these policies depending on the estimated position of the targets.

We consider Case 4 of  Section \ref{sec:formu} where the radar devotes all its resources to one target, and none to the other targets. That is, we assume a target priority vector of $\price = [1,0,0,0]$. In this case, the first target is allocated a Kalman filter, with all the other targets allocated {measurement-free} Kalman predictors. Since the threshold parametrization vectors depend on the target's state and measurement models, the first target $l=a$ has a unique parameter vector, where targets $l \neq a$ all have the same parameter vectors. Also, $\underline{\theta}^l = \theta$, for all $l \in\{1,2,...,L\}$.

We chose $\alpha^1 = 0.25$, $\alpha^2=\alpha^3=\alpha^4 = 0$, $\beta^1 = 0.25$, $\beta^2 = \beta^3 = \beta^4 = 1$ 
in stopping cost (\ref{eq:avgp}) (average mutual information difference stopping cost). The parametrized
policy considered was  $\mu_{\theta}(P^{a},P^{-a})$   in (\ref{eq:covpar}).
The optimal parametrized policy was computed using Algorithm \ref{alg1} at each of the $72$ locations on the GMTI sensor track. 
As the GMTI platform orbits the target region, we switch between these parametrized policy vectors, thus continually changing the adopted tracking policy. 
We implemented the following macro-manager:
$a = \argmax_{l = 1,...,L}\left\{ \log|P^l|\right\}$.
The priority vector was chosen as $\price^a = 1$ and $\price^l = 0$ for all $l\neq a$. Figure \ref{fig:persist_logdet}.
shows log-determinants of each of the targets' error covariance matrices
over multiple macro-management tracking cycles. 

\begin{figure}[p]
\centerline{\includegraphics[scale=0.4]{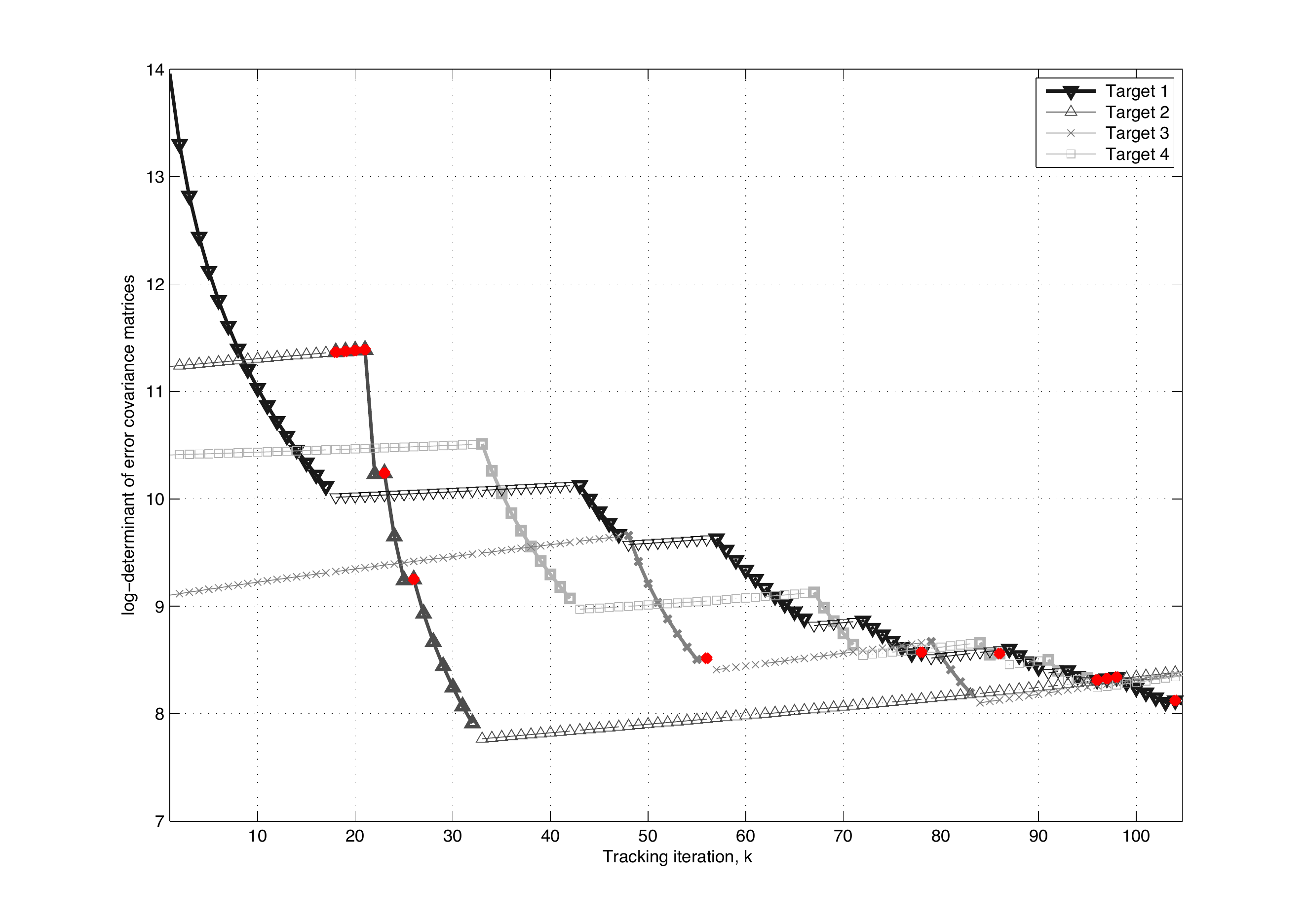}}
\caption{Plot of log-determinants for each target over multiple scheduling intervals.
On each scheduling interval, a Kalman filter is deployed to track one target and Kalman
predictors track the remaining 3 targets. The bold line  corresponds to the target
allocated the Kalman filter by the micro-manager in each scheduling interval.
Initially  a Kalman filter is deployed on target $l = 1$. Data points marked in red indicate missing observations.}\label{fig:persist_logdet}
\end{figure}

\section{Conclusions}
\label{sec:conclusion}
This paper considers  a sequential detection problem with mutual information
stopping cost. 
Using lattice programming
we
prove that the optimal policy 
 has a monotone structure in terms of the covariance estimates (Theorem \ref{thm:main}).
 The proof involved showing monotonicity of the Riccati and 
Lyapunov equations (Theorem \ref{thm:bitmead}).  Several examples of
parametrized decision policies that  satisfy this monotone structure  were given.
A simulation-based adaptive filtering algorithm (Algorithm (\ref{alg1})) was given to estimate the parametrized policy.
The sequential detection problem was illustrated in a GMTI radar scheduling problem with numerical
examples.

\appendix

\section{Proof of Theorem \ref{thm:main} and Theorem \ref{thm:bitmead}} \label{sec:proof}
This appendix presents the proof of the main result Theorem \ref{thm:main}. 
Appendix  \ref{app:prelim}  presents the value iteration algorithm and supermodularity that will be used as the basis
of the inductive proof.  
The proof of Theorem 1 in Appendix
\ref{app:th1} 
uses lattice programming \cite{Top98}
and depends on certain
monotone properties of the Kalman filter Riccati and Lyapunov equations. These properties
are proved in Theorem \ref{thm:bitmead} in Appendix \ref{sec:bitmead}.

\subsection{Preliminaries}  \label{app:prelim}
We first rewrite Bellman's equation  (\ref{eq:valuefcn}) in a form that is  suitable
for our analysis. 
Define 
\begin{align} 
 C(P) &= \cop - 
 \Cb(P) 
 +  \sum_{z^a,z^{-a}} \Cb(\R(P^a,z^a),\L(P^a),\R(P^{-a},z^{-a}),\L(P^{-a}))
 q_{z^a}q_{z^{-a}},
 \nonumber\\
  V(P) &= \Vb(P) -  \Cb(P), \quad \text{ where } P = (P^a,\bP^a,P^{-a},\bP^{-a}), \label{eq:coord} \\
  q_{z^l} &= \begin{cases} 
  p_d^l, & \text{ if } z^l \neq \emptyset, \\
   1 - p_d^l , &\text{ otherwise, } \end{cases}, l = 1,\ldots,L. \nonumber
  \end{align}
  In (\ref{eq:coord}) we have assumed that the missed observation events in (\ref{eq:obs}) are
  statistically independent between targets, and so
   $q_{z^{-a}} = \prod_{l\neq a} q_{z^l}$. Actually, the results below hold for any joint distribution 
 of
   missed observation events (and therefore allow these events to be dependent between targets). For
 notational convenience we assume (\ref{eq:obs}) and independent missed observation events.
   
Clearly $V(\cdot)$ and optimal decision policy $\mu^*(\cdot)$ satisfy Bellman's   equation 
\begin{align}
V(P) &=  \min \{\Q(P,1), \Q(P,2)\},\quad
 \mu^*(P) = \argmin_{u \in \{1,2\}} \bigl\{ \Q(P,u)\bigr\},
  \label{eq:costdef}  \\
 \Q(P,1) &= 0, \quad 
 \Q(P,2) = C(P) + \sum_{z^a,z^{-a}} V(\left(\R(P^a,z),\L(P^a),\R(P^{-a},z^{-a}),\L(P^{-a}) \right)q_{z^a} q_{z^{-a}}.
 \nonumber
  \end{align}
  Our goal is to characterize the stopping set defined as
  $$\Rs = \{P: \Q(P,1) \leq \Q(P,2)\} = \{P:  \Q(P,2)\geq 0 \}= \{ P: \mu^*(P) = 1\}.$$
Since the $\argmin$ function is translation invariant (that is, $\argmin_u f(P,u) = \argmin_u (f(P,u) + h(P))$ for any functions $f$ and $h$),   both the stopping set $\Rs$ and optimal policy $\mu^*$ in these
new coordinates are identical to those in the original coordinate system (\ref{eq:stopset}), (\ref{eq:valuefcn}).

\noindent {\em Value Iteration Algorithm}:
The value iteration algorithm  will be used
to construct a proof of Theorem \ref{thm:main} by mathematical induction.
Let $k=1,2,\ldots,$ denote iteration number. The   value iteration
algorithm is  a fixed point iteration of 
Bellman's equation and proceeds as follows:
\begin{align}\nonumber
V_0(P) &= -\Cb(P),
 \quad V_{k+1}(P)= \min_{u \in \{1,2\}} \Q_{k+1}(P,u), \\
\mu^*_{k+1}(P)&= \argmin_{u \in \{1,2\}} \Q_{k+1}(P,u), \quad 
\text{ where } \Q_{k+1}(P,1) = 0,
\nonumber\\
 \Q_{k+1}(P,2)& =  C(P) 
+ \sum_{z^a,z^{-a}}  V_k\left( \R(P^a,z^a),\R(P^{-a},z^{-a}) \right) q_{z^a}q_{z^{-a}} .
 \label{eq:vi}
\end{align}

{\em Submodularity}:
Next, we define the key concept of submodularity \cite{Top98}.  While it can be defined on general lattices
with an arbitrary partial order, here
 we restrict the definition to the posets  $[\pdf,\gr]$ and $[\R_+^m,\og]$, where the partial orders $\gr$ and $\og$ were defined above.

\begin{definition}[Submodularity and Supermodularity \cite{Top98}] \label{def:supermod}  A scalar function $\Q(P,u)$ is
submodular in $P^a$ if
 $$\Q(P^a,\bP^a,P^{-a},\bP^{-a},2) - \Q(P^a,\bP^a,P^{-a},\bP^{-a},1) \leq \Q(R^a,\bP^a,P^{-a},\bP^{-a},2) - \Q(R^a,\bP^a,P^{-a},\bP^{-a},1),$$ for $P^a \gr R^a$.
 $\Q(\cdot,\cdot,u)$ is supermodular if 
$-\Q(\cdot,\cdot,u)$ is submodular.
A scalar function  $\Q(P^a,\bP^a,P^{-a},\bP^{-a},u)$ is
sub/supermodular in each component of $P^{-a}$ if it is sub/supermodular in each component $P^{l}$, $l \neq a$.     An identical definition holds with respect to $\Q(\eig^a,\eig^{-a},u)$ on  $[\R_+^m,\og]$.
\end{definition}


The most important feature of a supermodular (submodular) function  $f(x,u)$ is that
$\argmin_u f(x,u)$  decreases (increases) in its argument $x$, see \cite{Top98}. This is summarized in the following result.

\begin{theorem}[\cite{Top98}] \label{res:monotone}
\label{res:supermod} Suppose $\Q(P^a,\bP^a,P^{-a},\bP^{-a},u)$ is submodular in $P^a$, submodular in 
$\bP^{-a}$, 
supermodular in $P^{-a}$ and supermodular in $\bP^a$.
Then {there} exists a $\mu^*(P^a,\bP^a,P^{-a},\bP^{-a}) = \argmin_{u\in \{1,2\}} \Q(P^a,\bP^a,
P^{-a},\bP^{-a},u)$, that  is increasing
in $P^a$, decreasing in $\bP^a$, increasing in $\P^{-a}$  and decreasing in $P^{-a}$.
\end{theorem}

Next we state a well known result (see \cite{AM79} for proof) 
that the evolution of the covariance matrix in the Lyapunov and Riccati equation are monotone.

\begin{lemma}[\cite{AM79}]\label{lem:1} $\mathcal{R}(\cdot)$ and $\L(\cdot)$ are   monotone operators on the poset
$[\pdf,\gr]$. That is,
if $P_1 \gr{P}_2$, then $\L(P_1) \gr \L({P_2})$ and  for all $z$,
$\mathcal{R}(P_1,z) \gr \mathcal{R}({P}_2,z)$.
\end{lemma}

Finally, we present the following lemma  which states that the stopping costs (stochastic observability)
are monotone in the covariance matrices. The proof of this lemma depends on Theorem \ref{thm:bitmead},
the proof
of which is given in Appendix \ref{sec:bitmead} below.

\begin{lemma} \label{lem:bob} For $\Cb(P)$ in Case 1 (\ref{eq:maxp}), Case 2  (\ref{eq:minp}) and Case 3
(\ref{eq:avgp}), the cost 
 $C(P^a,\bP^a,P^{-a},\bP^{-a})$ defined in (\ref{eq:coord}) is decreasing
 in $ P^a$, $  \bP^{-a}$, and increasing in 
 $  P^{-a}$, 
$  \bP^a$. (Case 4 is a special case when $\alpha_l = 0$ for all $l \in \{1,\ldots,L\}$.)
\end{lemma}

\begin{proof}
For Case 1 and Case 2 let $\lstar = \arg\max_{l\neq a} \left[\alpha^l \log|\bP_k^l|
-  \beta^l \log|P_k^l| \right]$ or $\lstar = \arg\min_{l\neq a} \left[\alpha^l \log|\bP_k^l|
-  \beta^l \log|P_k^l| \right]$, respectively.
From (\ref{eq:coord}) with $|\cdot|$ denoting determinant,
\beq
C(P) = \cop - \alpha^a \log \frac{|\L(\bP^a)|}{|\bP^a|} +\beta^a \sum_{z^a} \log \frac{|\R(P^a,z^a)|}{|P^a|} q(z^a)
+ \alpha^{\lstar} \log \frac{|\L(\bP^\lstar)|}{|\bP^\lstar|} -  \beta^\lstar \log \frac{|\R(P^\lstar,z^\lstar)|}{|P^\lstar|}  q(z^\lstar).
\eeq
For Case 3, 
\beq
C(P) = \cop - \alpha^a \log \frac{|\L(\bP^a)|}{|\bP^a|} +\beta^a \sum_{z^a} \log \frac{|\R(P^a,z^a)|}{|P^a|} q(z^a)
+ \sum_{l\neq a} \alpha^l \log \frac{|\L(\bP^l)|}{|\bP^l|} - \sum_{l\neq a} \beta^l \log \frac{|\R(P^l,z^l)|}{|P^l|}  q(z^l).
\eeq
Theorem \ref{thm:bitmead}  shows that
$\frac{|\L(\bP^l)|}{|\bP^l|}$ and $\frac{|\R(P^l,z^l)|}{|P^l|}$ are decreasing in $\bP^l$ and $P^l$ for all $l$.
\end{proof}

\subsection{Proof of Theorem \ref{thm:main}} \label{app:th1}

\begin{proof}
The proof is by induction on the value iteration algorithm (\ref{eq:vi}).
Note $V_0(P)$ defined in (\ref{eq:vi}) is decreasing in $P^a,\bP^{-a}$ and increasing in $P^{-a},\bP^a$ via
Lemma \ref{lem:1}.

Next assume $V_k(P)$ is decreasing in $P^a,\bP^{-a}$ and increasing in $P^{-a},\bP^a$.
Since $\R(P^a,y)$, $\L(\bP^a)$, $\R(P^{-a},y^{-a})$ and $\L(\bP^{-a})$ are monotone increasing in $P^a$,
$\bP^a$, $P^{-a}$ and $\bP^{-a}$, it follows
that  the term $V_k\left( \R(P^a,z^a),\L(\bP^a),\R(P^{-a},z^{-a}),\L(\bP^{-a}) \right) q_{z^a}q_{z^{-a}}$
is decreasing in $P^a, \bP^{-a}$ and increasing in  $P^{-a},\bP^a$ in (\ref{eq:vi}).
Next, it follows from Lemma \ref{lem:bob} that  $C(P^a,,\bP^a,P^{-a},\bP^{-a})$  is decreasing in $P^a, \bP^{-a}$ and increasing in  $P^{-a},\bP^a$.
Therefore from (\ref{eq:vi}), $\Q_{k+1}(P,2)$ inherits this property.
Hence $V_{k+1}(P^a,P^{-a})$ is decreasing in $P^a, \bP^{-a}$ and increasing in  $P^{-a},\bP^a$.
 Since
value iteration converges pointwise, i.e, $V_{k}(P)$ pointwise $V(P)$, it follows that
 $V(P)$ is decreasing in $P^a, \bP^{-a}$ and increasing in  $P^{-a},\bP^a$.

Therefore,
 $\Q(P,2)$ is  decreasing in $P^a, \bP^{-a}$  and increasing in $P^{-a},\bP^{a}$. This implies
 $\Q(P,u)$ is submodular in $(P^a,u)$, submodular in $(\bP^{-a},u)$, supermodular in $(P^{-a},u)$ and supermodular
 in $(\bP^{a},u)$. Therefore, from Theorem \ref{res:monotone}, there exists a version
 of $\mu^*(P)$ that  is  increasing in $P^a,\bP^{-a}$ and decreasing in $P^{-a},\bP^a$.
   \end{proof}

\subsection{Proof of  Theorem \ref{thm:bitmead}} \label{sec:bitmead}
We start with the following lemma.
\begin{lemma}
If matrices $X$ and $Z$ are invertible, then for conformable matrices $W$ and $Y$,
\begin{align}
\text{det}(Z)\text{det}(X+YZ^{-1}W) = \text{det}(X)\text{det}(Z+WX^{-1}Y).\label{eq:lem1}
\end{align}
\end{lemma}

\begin{proof}
The Schur complement formulae  applied to $\left( {X\atop -W} {Y\atop Z} \right)$ yields,
\begin{multline}
\nonumber \left(
\begin{array}{ccc}
I & YZ^{-1}\\
0 & I \end{array}
\right)\left(
\begin{array}{ccc}
X+YZ^{-1}W & 0\\
0 & Z \end{array}
\right)\left(
\begin{array}{ccc}
I & 0\\
-Z^{-1}W & I \end{array}
\right) \\ = \left(
\begin{array}{ccc}
I & 0\\
-WX^{-1} & I \end{array}
\right)\left(
\begin{array}{ccc}
X & 0\\
0 & Z+WX^{-1}Y \end{array}
\right)\left(
\begin{array}{ccc}
I & X^{-1}Y\\
0 & I \end{array}
\right).
\end{multline}
Taking determinants yields (\ref{eq:lem1}).
\end{proof}




{\em Theorem \ref{thm:bitmead}(i)}:
Given positive definite matrices $Q$ and $P_1\succ P_2$ and arbitrary matrix $F$,
\begin{align} \frac{\text{det}(\L(P))}{\text{det}(P)} \text{ is decreasing in } P, \text{ or equivalently, }
\nonumber \frac{\text{det}(FP_1F^T+Q)}{\text{det}(P_1)}<\frac{\text{det}(FP_2F^T+Q)}{\text{det}(P_2)}
\end{align}

\begin{proof}
Applying  (\ref{eq:lem1}) with $[X,Y,W,Z] = [Q,F,F^T,P^{-1}]$,
\begin{align}
\frac{\text{det}(FPF^T+Q)}{\text{det}(P)}=\text{det}(P^{-1} + F^TQ^{-1}F)\text{det}(Q).\label{eq:thm1}
\end{align}
Since $P_1\succ P_2\succ 0$, then $0\prec P_1^{-1}\prec P_2^{-1}$ and thus $0\prec P_1^{-1} + F^TQ^{-1}F\prec P_2^{-1}+F^TQ^{-1}F$. 
Since positive definite dominance implies dominance of determinants, it follows that
\begin{align}
\nonumber \text{det}(P_1^{-1} + F^TQ^{-1}F)<\text{det}(P_2^{-1} + F^TQ^{-1}F).
\end{align}
Using (\ref{eq:thm1}), the result follows.
\end{proof}


{\em Theorem \ref{thm:bitmead}(ii)}:
Given positive definite matrices $Q,R$  and arbitrary matrix $F$,
$ \frac{\text{det}(\R(P,z))}{\text{det}(P)}$ is decreasing in $P$.
That is for  $P_1\succ P_2$
\begin{multline}
\frac{\text{det}(FP_1F^T - FP_1H^T(HP_1H^T+R)^{-1}HP_1F^T+Q)}{\text{det}(P_1)} \\ <\frac{\text{det}(FP_2F^T - FP_2H^T(HP_2H^T+R)^{-1}HP_2F^T+Q)}{\text{det}(P_2)}\label{eq:thm2}
\end{multline}

\begin{proof}
Using the  matrix inversion lemma
$ (A+BCD)^{-1} = A^{-1} - A^{-1}B(C^{-1}+DA^{-1}B)^{-1}DA^{-1}$,
\begin{align}
\nonumber (P^{-1}+H^TR^{-1}H)^{-1} = P-PH^T(HPH^T+R)^{-1}HP,\\
\nonumber F(P^{-1}+H^TR^{-1}H)^{-1}F^T+Q = FPF^T-FPH^T(HPH^T+R)^{-1}HPF^T+Q \\
\text{det}(FPF^T-FPH^T(HPH^T+R)^{-1}HPF^T+Q) = \text{det}(Q+F(P^{-1}+H^TR^{-1}H)^{-1}F^T).\label{eq:det1}
\end{align}
Applying the identity (\ref{eq:lem1}) with $[X,Y,W,Z] = [Q,F,F^T,P^{-1}+H^TR^{-1}H]$ we have,
\beq
\text{det}(P^{-1}+H^TR^{-1}H)\text{det}(Q+F(P^{-1}+H^TR^{-1}H)^{-1}F^T) = \text{det}(Q)\text{det}(P^{-1}+H^TR^{-1}H+F^TQ^{-1}F).\label{eq:det2}
\eeq
Further, using (\ref{eq:lem1}) with $[X,Y,W,Z] = [P^{-1},H^T,H,R]$, we have,
\begin{align}
\text{det}(P^{-1}+H^TR^{-1}H) = \text{det}(P^{-1})\text{det}(R+HPH^T)/\text{det}(R)\label{eq:det3}
\end{align}
Substituting  (\ref{eq:det3}) into (\ref{eq:det2}),
\begin{multline}
 \text{det}(P^{-1})\text{det}(R+HPH^T)\text{det}(Q+F(P^{-1}+H^TR^{-1}H)^{-1}F^T)  \\ = \text{det}(Q)\text{det}(P^{-1}+H^TR^{-1}H+F^TQ^{-1}F)\text{det}(R)
\end{multline}
\begin{align}
\frac{\text{det}(Q+F(P^{-1}+H^TR^{-1}H)^{-1}F^T)}{\text{det}(P)}=\frac{ \text{det}(Q)\text{det}(P^{-1}+H^TR^{-1}H+F^TQ^{-1}F)\text{det}(R)}{\text{det}(R+HPH^T)}.\label{eq:det4}
\end{align}
From (\ref{eq:det4}) and (\ref{eq:det1})
\begin{align}
\frac{\text{det}(FPF^T-FPH^T(HPH^T+R)^{-1}HPF^T+Q)}{\text{det}(P)}=\frac{ \text{det}(Q)\text{det}(P^{-1}+H^TR^{-1}H+F^TQ^{-1}F)\text{det}(R)}{\text{det}(R+HPH^T)}.
\end{align}
We are now ready to prove the result. Since $P_1\succ P_2\succ 0$,
\begin{itemize}
\item $0\prec P_1^{-1}+H^TR^{-1}H+F^TQ^{-1}F\prec P_2^{-1}+H^TR^{-1}H+F^TQ^{-1}F$,
\item $\text{det}(P_1^{-1}+H^TR^{-1}H + F^TQ^{-1}F)<\text{det}(P_2^{-1}+H^TR^{-1}H+F^TQ^{-1}F)$,
\item $R+HP_1H^T\succ R+HP_2H^T\succ 0$,
\item $\text{det}(R+HP_1H^T)>\text{det}(R+HP_2H^T)$.
\end{itemize}
Therefore,  (\ref{eq:thm2}) follows from the following inequality
\begin{align*}
\frac{ \text{det}(Q)\text{det}(P_1^{-1}+H^TR^{-1}H+F^TQ^{-1}F)\text{det}(R)}{\text{det}(R+HP_1H^T)}<\frac{ \text{det}(Q)\text{det}(P_2^{-1}+H^TR^{-1}H+F^TQ^{-1}F)\text{det}(R)}{\text{det}(R+HP_2H^T)}
\end{align*}
\end{proof}

\bibliographystyle{IEEEbib}
\bibliography{$HOME/vikramk/styles/bib/vkm}
\end{document}